\documentclass{gtmon_a}
\pdfoutput=1

\usepackage{xypic}
\usepackage{stmaryrd}
\input diagxy

\proceedingstitle{Proceedings of the Nishida Fest (Kinosaki 2003)}
\conferencestart{28 July 2003}
\conferenceend{8 August 2003}
\conferencename{International Conference in Homotopy Theory}
\conferencelocation{Kinosaki, Japan}

\editor{Matthew Ando}
\givenname{Matthew}
\surname{Ando}

\editor{Norihiko Minami}
\givenname{Norihiko}
\surname{Minami}

\editor{Jack Morava}
\givenname{Jack}
\surname{Morava}

\editor{W Stephen Wilson}
\givenname{W Stephen}
\surname{Wilson}

\title{Cohomology operations and algebraic geometry}

\author{Simone Borghesi}
\givenname{Simone}
\surname{Borghesi}
\address{Universit\'a degli Studi di Milano-Bicocca\\
Dipartimento di Matematica e Applicazioni\\\newline
via Cozzi 53\\
20125 Milano\\
Italy}
\email{mandu2@libero.it}
\urladdr{}

\volumenumber{10}
\issuenumber{}
\publicationyear{2007}
\papernumber{4}
\startpage{75}
\endpage{115}

\doi{}
\MR{}
\Zbl{}

\arxivreference{}


\received{2 December 2004}
\revised{9 September 2005}
\accepted{}
\published{29 January 2007}
\publishedonline{29 January 2007}
\proposed{}
\seconded{}
\corresponding{}
\version{}

\makeatletter
\def\cnewtheorem#1[#2]#3{\newtheorem{#1}{#3}[section]
\expandafter\let\csname c@#1\endcsname\c@theorem}

\let\xysavmatrix\xymatrix
\def\xymatrix{\disablesubscriptcorrection\xysavmatrix}
\AtBeginDocument{\let\bar\wbar\let\tilde\widetilde\let\hat\widehat\let\wcheck\check}

\renewcommand{\to}{\rightarrow}
\let\barrsquare\square
\let\square\undefined

\newtheorem{theorem}{Theorem}[section]
\cnewtheorem{proposition}[theorem]{Proposition}
\cnewtheorem{lemma}[theorem]{Lemma}
\cnewtheorem{corollary}[theorem]{Corollary}
\newtheorem*{conjecture}{Conjecture}
\theoremstyle{remark}
\cnewtheorem{definition}[theorem]{Definition}
\cnewtheorem{examples}[theorem]{Examples}
\cnewtheorem{remark}[theorem]{Remark}

\makeatother  

\newcommand{\p}{{\mathbb Z}/p}
\newcommand{\au}{{\mskip0mu\underline{\mskip-0mu a \mskip-2mu}\mskip2mu}}
\newcommand{\am}{{\mathcal{A}}^{*,*}}
\newcommand{\hb}{H^{*,*}}
\newcommand{\act}{\stackrel{\mathcal{A}}{\cdot}}
\newcommand{\hkp}{\mathcal{H}_\bullet(k)}
\newcommand{\xa}{\mathcal{X}_{\au}}
\newcommand{\xar}{\tilde{\mathcal{X}}_{\au}}
\newcommand{\pro}{{\mathbb P}^1_k}
\newcommand{\due}{\mathbb{Z}/2}
\newcommand{\dual}{\mathcal{A}_{*,*}}
\newcommand{\affs}{{\mathbb A}^1}
\newcommand{\aff}{{\mathbb A}^1_k}
\newcommand{\affx}{{\mathbb A}^1_X}
\newcommand{\nshsets}{\Shv(\Sm/k)_{\Nis}}

\makeop{Spec}
\makeop{Cof}
\makeop{Id}
\makeop{id}
\makeop{Sing}
\makeop{Hom}
\makeop{sing}
\makeop{op}
\makeop{Sets}
\makeop{Ab}
\makeop{Ch}
\makeop{Nis}
\makeop{Th}
\makeop{Shv}
\makeop{Sch}
\makeop{Sm}
\makeop{Funct}
\makeop{Sq}
\makeop{tr}
\makeop{red}
\makeop{top}
\makeop{Newt}
\makeop{et}
\makeop{Pic}
\makeop{Top}
\makeop{Tot}
\makeop{Zar}
\makeop{fl}
\makeop{colim}
\newcommand{\chara}{\operatorname{char}}

\begin{document}

\begin{abstract}
The manuscript is an overview of the motivations and foundations lying
behind Voevodsky's ideas of constructing categories similar to the
ordinary topological homotopy categories. The objects of these categories
are strictly related to algebraic varieties and preserve some of their
algebraic invariants.
\end{abstract}

\maketitle


\section{Introduction}

This manuscript is based on a ten hours series of seminars I delivered 
in August of 2003 at
the Nagoya Institute of Technology as part of the workshop on homotopy
theory organized by Norihiko Minami and following the Kinosaki 
conference in honor of Goro Nishida. 
One of the most striking applications of homotopy
theory in ``exotic'' contexes is Voevodsky's proof of the Milnor 
Conjecture. This conjecture can be reduced to statements about
algebraic varieties and ``cohomology theories'' of algebraic
varieties. These contravariant functors are called {\it motivic
  cohomology} with coefficients in abelian groups $A$. Since they
share several properties with singular cohomology in
classical homotopy theory, it is reasonable to expect ``motivic
cohomology operations'' acting naturally on these cohomology theories. 
By assuming the
existence of certain motivic Steenrod operations and guessing their right
degrees, Voevodsky was able to prove the Milnor 
Conjecture. This strategy reduced the complete proof of the conjecture
to the construction of these operations and to an appropriate category
in whcih motivic cohomology is a ``representable''. 
In homotopy theory there are several ways of doing this. We now know
two ways of obtaining such operations on motivic cohomology: one is
due to P Brosnan \cite{brosnan} and the other to V Voevodsky 
\cite{voe-oper}. The latter approach follows a systematic developement of
homotopy categories containing algebraic information of the underlying
objects and it is the one we will discuss in this manuscript. 
It turns out that, if we think of an algebraic variety as 
something like a topological space with an algebraic structure
attached to it, it makes sense to try to construct homotopy categories
in which objects are algebraic varieties, as opposed to just topological
spaces, and, at the same time motivic cohomology representable. In
the classical homotopy categories such representability always holds
for any cohomology theory because of Brown Representability Theorem. 
If such exotic homotopy categories existed in our algebraic setting, 
they would tautologically contain all the algebraic invariants
detected by the cohomology theories represented. 

This manuscript begins with a motivational part constituted by an 
introduction to Voevodsky's reduction steps of the Milnor Conjecture. 
By the end, the conjecture is reduced to the
existence of motivic Steenrod operations and of an (unstable) homotopy
category of schemes. In the second section
we will discuss some of the main issues  involved on the construction
of the homotopy category of schemes and in the proof of the
representability of motivic cohomology, assuming perfectness of the 
base field. Another reference for
representability of motivic cohomology are the lecture notes by Voevodsky
and Deligne \cite{voe-del}. We will 
focus particularly on the identifications between 
objects in the localized categories. 

For the beautiful memories of the period spent in Japan,
which included the workshop at the Nagoya Institute of Technology, I
am greatly indebted to Akito Futaki and to Norihiko Minami.

\section{Motivic Steenrod operations in the Milnor Conjecture}

Throughout this section, the word scheme will refer to a separable scheme of
finite type over a field $k$. Alternatively, we may consider a scheme
to be an algebraic variety over a field $k$.
By Milnor Conjecture we mean the statement known as  
Bloch--Kato Conjecture at the prime $p=2$ (for more information
about the origin of this conjecture, see the introduction of the paper
\cite{on2tors} by Voevodsky). This conjecture asserts:
\begin{conjecture}[Bloch--Kato]
Let $k$ be a field of characteristic different from a prime number
$p$. Then the norm residue homomorphism
\begin{equation}
N\co K_n^M(k)/(p)\to \mathbb{H}_{\et}^n(\Spec k,\mu_p^{\otimes n}) 
\end{equation}
is an isomorphism for any nonnegative integer $n$.
\end{conjecture}
We assume the reader to have mainly a homotopy theoretic background,
therefore we will occasionally include some descriptions of objects
used in algebraic geometry. The {\it Milnor $K$--theory} 
is defined as the graded ring
$$\{K_0^M(k),K_1^M(k), K_2^K(k),\ldots\}$$ where $K_n^M(k)$ is the
quotient of the group $k^*\otimes_{\mathbb{Z}}\stackrel{n}{\cdots}\otimes_{
\mathbb{Z}}k^*$ by the subgroup generated by the elements
$a_1\otimes\cdots\otimes a_n$ where $a_i+a_{i+1}=1$ for some $i$. It
is useful to mention that, in the literature, when dealing with Milnor
$K$--theory, the multiplicative group $k^*$ of the field $k$ is written
additively: for instance the element in $K_n^M(k)$ 
represented by $a_1\otimes\cdots\otimes a_m^c\otimes\cdots a_n$ belongs to
the subgroup $c\cdot K_n^M(k)$, being equal to $a_1\otimes\cdots
c\cdot a_m\otimes\cdots a_n$ where $c\cdot a_m$ is in $K_1^M(k)=(k^*,+)$. 
The group lying as target of the norm residue homomorphism is a
(hyper)cohomology group of the algebraic variety $\Spec k$. As first
approximation we may think of it as being a sort of ordinary
cohomology functor $H^*(-, A)$ on algebraic varieties 
in which, instead of having the abelian group $A$ as the only input, we
have two inputs: the {\it Grothendieck topology} ({\it et}=\'etale in this
case), and a {\it complex of sheaves of abelian groups} for the 
topology considered as
``coefficients'' of the cohomology (the \'etale complex of sheaves 
$\mu_p^{\otimes n}$ in the statement of the conjecture). The complex
of sheaves $\mu_p$ is zero at any degree except in degree zero where
it is the sheaf that associates the elements $f\in
\mathcal{O}(X)$ such that $f^p=1$ to any smooth scheme of
finite type $X$ over a field $k$. $\mu_p^{\otimes n}$ is the $n$--fold
tensor product of $\mu_p$ in the derived category (one can construct
this monoidal structure in a similiar way as in the derived
category of complexes of abelian groups). As we would expect, the
cohomology theory $\mathbb{H}_{\et}^*(-, \mu_p^{\otimes *})$
is endowed of a commutative ring structure given by the {\it cup product}.  
The norm residue homomorphism is defined as
$$N(\overline{a_1\otimes\cdots\otimes a_n})=\delta(a_1)\cup\cdots\cup
\delta(a_n)\in \mathbb{H}_{\et}^n(\Spec k, \mu_p^{\otimes n}),$$
where $\delta$ is the coboundary operator 
$\mathbb{H}_{\et}^0(\Spec k,\mathbb{G}_m)\to
\mathbb{H}_{\et}^1(\Spec k,\mu_p)$ associated to the short exact
(for the \'etale topology) 
sequence of sheaves
\begin{equation}
\xymatrix{0\ar[r]& \mu_p\ar[r]&\mathbb{G}_m\ar[r]^{\hat{}p}&
  \mathbb{G}_m\ar[r]& 0}
\end{equation} 
We recall that {\it the multiplicative group} sheaf $\mathbb{G}_m$ is
defined as $\mathbb{G}_m(X)=\mathcal{O}(X)^*$ for any smooth scheme
$X$ and that
$H^0(X,\mathcal{F})=\mathcal{F}(X)$ for any sheaf $\mathcal{F}$.
$N$ defines indeed an homomorphism from $K_*^M(k)$
because $\delta(a)\cup \delta(1-a)=0$ in
$\mathbb{H}_{\et}^2(\Spec k,\mu^{\otimes 2})$ as a consequence of a 
result of Bass and Tate in \cite{bt}. 
\subsection{Reduction steps}

There are several reduction steps in Voevodsky's program to prove the
 general Bloch--Kato Conjecture before motivic cohomology operations 
are used. Firstly, the conjecture follows from another one:
\begin{conjecture}[Beilinson--Lichtenbaum, $p$--local version]
Let $k$ be a field and $w\geq 0$. Then 
\begin{equation}
\mathbb{H}_{\et}^{w+1}(\Spec k,\mathbb{Z}_{(p)}(w))=0
\end{equation}
\end{conjecture}

In the proof that this statement implies the Bloch--Kato conjecture we
begin to use the so called {\it motivic cohomology theory}
$H^{i,j}(-, A)$ with coefficients in an abelian group $A$.

\begin{proof}[Proof (that the Beilinson--Licthembaum Conjecture implies
the Bloch--Kato Conjecture)]
Let $H^{i,j}(X,\p)$ be the $(i,j)$th motivic
cohomology group of $X$ with
coefficients in $\p$ (see \fullref{motcoho}). Assuming that the
Beilinson--Lichtenbaum conjecture holds in degrees less or equal than $n$,
Voevodsky proved in \cite[Corollary~6.9(2)]{on2tors} that
$H^{i,j}(X,\p)\cong \mathbb{H}^{i}_{\et}(X,\p(j))$ for all $i\leq j\leq
n$. On the other hand, in Theorem 6.1 of the same paper, he shows that
$\mathbb{H}^{i}_{\et}(X,\p(j))\cong \mathbb{H}^i_{\et}(X,\mu_p^{\otimes
  j})$. We conclude by recalling that, 
if $i=j$ and $X=\Spec k$, the natural
homomorphism $K_i^M(k)/(p)\to H^{i,i}(\Spec k,\p)$
is an isomorphism (see Suslin--Voevodsky \cite[Theorem~3.4]{sus-voe}).
\end{proof}

The proof of the Beilinson--Lichtenbaum Conjecture is by induction on
the \, index $w$. For $w=0$ we have that
$$\mathbb{H}_{\et}^1(\Spec k,\mathbb{Z}(0))=\mathbb{H}_{\et}^1(\Spec k,
\mathbb{Z})=H^{1,0}(\Spec k,\mathbb{Z})=0$$
and in the case $w=1$, we know that
$\mathbb{Z}(1)=\mathbb{G}_m[-1]$, thus
$$\mathbb{H}_{\et}^2(\Spec k,\mathbb{Z}(1))=H^1_{\et}(\Spec k,\mathbb{G}_m)
=\Pic(\Spec k)=0$$

The first reduction step is that we can assume that
$k$ has no finite extensions of degree prime to $p$ (for the time
being such field will be called $p$--special). Indeed,
$\mathbb{H}_{\et}^{w+1}(\Spec k,\mathbb{Z}_{(p)}(w))$ injects in
$\mathbb{H}_{\et}^{w+1}(\Spec L,\mathbb{Z}_{(p)}(w))$ for any prime to
$p$ degree field extension $L/k$, because the composition
$$\xymatrix{\mathbb{H}_{\et}^{*}(\Spec k,\mathbb{Z}_{(p)}(*))
\ar[r]^{\mathrm{transfer}}&
\mathbb{H}_{\et}^{*}(\Spec L,\mathbb{Z}_{(p)}(*))
\ar[r]^{i^*} & \mathbb{H}_{\et}^{*}(\Spec k,\mathbb{Z}_{(p)}(*))}
$$
is multiplication by $[L:k]$, hence it is an isomorphism. Therefore,
letting $F$ to be the colimit over all the prime to $p$ field
extensions of $k$, to conclude it suffices to show the vanishing
statement for $F$. 

Secondly, if $k$ is $p$--special and $K_w^M(k)/(p)=0$, it is possible
to prove directly that
$\mathbb{H}_{\et}^{w+1}(\Spec k,\mathbb{Z}_{(p)}(w))=0$. 
Hence, to prove the Bloch--Kato Conjecture it suffices to prove the 
following statement.
\begin{theorem}\label{ridotto}
 For any $0\neq \{a_1,a_2,\ldots, a_w\}=\au\in K_w^M(k)/(p)$,
  there exists a field extension $k_{\au}/k$ with $k$ being $p$--special such
  that:
\begin{enumerate}
\item\label{a} $\au\in\ker i_*$, where $i_*\co K_w^M(k)/(p)\to
  K_w^M(k_{\au})/(p)$ is the induced homomorphism;
\item\label{b}
 $i^*\co \mathbb{H}^{w+1}_{\et}(\Spec k,\mathbb{Z}_{(p)}(w))\hookrightarrow
  \mathbb{H}_{\et}^{w+1}(k_{\au},\mathbb{Z}_{(p)}(w))$ is an injection. 
\end{enumerate} 
\end{theorem}
Surprisingly, the best candidates for such fields, we have knowledge of, 
are function fields of appropriate algebraic varieties. Indeed,
$k(a_i^{1/p})$ are all fields satisfying condition \eqref{a} for any
$1\leq i\leq w$, and condition \eqref{b} is precisely the complicated
one. The approach to handle \eqref{b} is to give a sort of
underlying ``algebraic variety'' structure to the field
$k_{\au}$. Voevodsky proved the Milnor Conjecture by showing that, if
$p=2$, we can take $k_{\au}$ to be the function field of the
projective quadric $Q_{\au}$ given by the equation 
\begin{equation}
\langle 1,-a_1\rangle\otimes\langle 1,-a_2\rangle\otimes\cdots\otimes
\langle 1,-a_{w-1}\rangle\oplus \langle -a_w\rangle=0
\end{equation}
with the convention that $$\langle a,b\rangle\otimes\langle c,d\rangle=
act_1^2+adt_2^2+bct_3^2+bdt_4^2$$ and $$\langle k_1,k_2,\ldots,
k_m\rangle \oplus \langle h\rangle=k_1t^2_1+k_2t_2^2+\cdots
+k_mt_m^2+ht_{m+1}^2$$ 
Such variety is known as {\it Pfister
  neighborhood} of the quadric
$\langle 1,-a_1\rangle\otimes
  \langle 1,-a_2\rangle\otimes\cdots\otimes\langle 1,-a_{w-1}\rangle$.
In order to prove this, Voevodsky used two results of Markus Rost about such
quadrics translated in a context in which it makes sense to use
homotopy theoretical tools on algebraic varieties and an argument
involving motivic cohomology operations. We are now going to examine
more carefully these techniques.

\begin{remark}
Recently Rost worked on finding substitute varieties for
Pfister neighborhoods at odd primes. His candidates are called
{\it norm varieties} and he uses in his program certain formulae
called {\it (higher) degree formulae} (see Rost \cite{rost-split,icm}
and Borghesi \cite{io-for-grado}) to prove the relevant
properties in order to fit in Voevodsky's framework. 
\end{remark}

At this stage, we will take for granted the existence of a localized
category $\hkp$ whose objects are pointed simplicial algebraic
varieties (or, more generally, pointed simplicial sheaves for the Nisnevich
topology on the site of smooth schemes over $k$). The localizing
structure is determined by $\hkp$ being the homotopy category
associated to an ($\aff$) model structure on the category of
simplicial sheaves. These categories will be discussed more
extensively later in the manuscript (see \fullref{puggia}). At
this point all we need to know is that
\begin{enumerate}
\item\label{a1} the objects of $\hkp$ include pointed simplicial
  algebraic varieties;
\item\label{b1} any morphism $f\co \mathcal{X}\to \mathcal{Y}$ in $\hkp$ 
can be  completed to a sequence 
\begin{equation}
\mathcal{X}\stackrel{f}{\longrightarrow}\mathcal{Y}\longrightarrow\mathcal{Z}
\longrightarrow\mathcal{X}\wedge S_s^1\stackrel{f\wedge
\id}{\longrightarrow} \mathcal{Y}\wedge S_s^1 \longrightarrow \cdots 
\end{equation}
inducing long exact sequences of sets, or abelian groups when
appropriate, by applying the functor $\Hom_{\hkp}(-,
\mathcal{W})$ for any $\mathcal{W}\in\hkp$;
\item\label{c1} for all integers $i$, $j$ and abelian group $A$, there exist 
objects $K(A(j),i)\in\hkp$ such that $\Hom_{\hkp}(X_+,
K(A(j),i))=H^{i,j}(X,A)$ for any smooth algebraic variety
$X$ ($X_+$ is the pointed object associated to $X$, that is $X\amalg
\Spec k$). Moreover, defining $H^{i,j}(\mathcal{X},
A)=\Hom_{\hkp}(\mathcal{X}, K(A(j),i))$ for any pointed simplicial
algebraic variety $\mathcal{X}$, we have that, if $k$ is a perfect
field, the following equalities hold (see Voevodsky
\cite[Theorem~2.4]{voe-oper}):
\begin{eqnarray}
H^{i+1,j}(X\wedge S_s^1,A) & \cong & H^{i,j}(X,A) \\
H^{i+2,j+1}(X\wedge\pro,A) & \cong & H^{i,j}(X,A)
\end{eqnarray}
where $\wedge$ is the usual categorical smash product defined as
$\mathcal{X}\times \mathcal{Y}/\mathcal{X}\vee\mathcal{Y}$. 
\end{enumerate}  
Perfectness of the base field is not restrictive for the purpose of
the Bloch--Kato Conjecture since such conjecture holding on
characteristic zero fields implies the same result for 
fields of characteristic different from $p$. 
Let $X$ be a variety; denote by $\wcheck{C}(X)$ the simplicial
variety given by $\wcheck{C}(X)_n=X\times\smash{\stackrel{n+1}{\cdots}} \times
X$ with projections and diagonals as structure maps. The main feature
of such simplicial variety is that it becomes simplicially equivalent
to a point (ie $\Spec k$) if $X$ has a rational point $x\co \Spec k\to
X$, a contracting homotopy being 
\begin{equation}
\id\times x\co X\times\stackrel{i}{\cdots}\times X\to X\times
\stackrel{i+1}{\cdots}\times X
\end{equation}
By \eqref{b1}, the canonical map $\wcheck{C}(X)_+\to \Spec k_+$ can be 
completed to a sequence
\begin{equation}
\label{vai}
\wcheck{C}(X)_+\to \Spec k_+\to \tilde{C}(X)\to \wcheck{C}(X)_+\wedge
S_s^1\to\cdots 
\end{equation}
for some object $\tilde{C}(X)$ of $\mathcal{H}_\bullet(k)$.

We now assume $p=2$ as in
that case things are settled and we let $\xa$ to be the simplicial
smooth variety $\wcheck{C}(R_{\au})$, $R_{\au}$
being the Pfister neighborhood associated to the symbol $\au$. 
Part \eqref{a} of
\fullref{ridotto} for $k_{\au}=k(Q_{\au})$ is consequence of
a standard property of Pfister quadrics: if $R_{\au}$ has a rational
point on a field extension $L$ over $k$, then $\au$ is in the kernel
of the map $i^*\co K^M_*(k)/(2)\to K^M_*(L)/(2)$ (see Voevodsky
\cite[Proposition~4.1]{on2tors}). Part \eqref{b} follows 
from two statements: 
\begin{enumerate}
\renewcommand{\labelenumi}{(\roman{enumi})}
\item there exists a surjective map $H^{w+1,w}(\xa,\mathbb{Z}_{(2)})\to 
\ker i^*$, and 
\item $H^{w+1,w}(\xa,\mathbb{Z}_{(2)})=0$. 
\end{enumerate}
The first statement is already
nontrivial and uses various exact triangles in certain triangulated
categories, one of which is derived by a result of Rost on the Chow
motive of $R_{\au}$, and a
very technical argument (not mentioned by Voevodsky \cite{on2tors})
involving commutativity of the functor $\mathbb{H}_{\et}^*(-,\mathbb{Z}_{(2)}(*))$ with limits. The reason for introducing the
simplicial algebraic variety $\xa$ is to have an object sufficiently
similar to $\Spec k$, and at the same time sufficiently different to
carry homotopy theoretic information. The similarity to $\Spec k$ 
is used to prove (i), whereas the homotopy information plays a
fundamental role in showing that 
\begin{theorem}\label{zero}
$H^{w+1,w}(\xa,\mathbb{Z}_{(2)})=0$   
\end{theorem}

\begin{proof}
Let $\xar$ be $\tilde{C}(R_{\au})$.
Since $R_{\au}$ has points of degree two $\Spec E\to R_{\au}$ over $k$,
we have that $(\xar)_E\cong \Spec E$, because $f_E$ (the base change of
the structure map $f\co \xar\to \Spec k$ over $\Spec E$) is a simplicial 
weak equivalence. By a transfer argument, we
see that $2H^{*,*}(\xar,\mathbb{Z}_{(2)})=0$. Moreover, property
\eqref{c1} of $\hkp$ implies that
$H^{w+1,w}(\xa,\mathbb{Z}_{(2)})\cong
\tilde{H}^{w+2,w}(\xar,\mathbb{Z}_{(2)})$, because
$H^{i,j}(\Spec k,\mathbb{Z})=0$ if $i>j$. Thus, it suffices to show
that the image of the reduction modulo $2$ map
\begin{equation}\label{mod2}
\tilde{H}^{w+2,w}(\xar,\mathbb{Z}_{(2)})\to
\tilde{H}^{w+2,w}(\xar,\mathbb{Z}/(2))
\end{equation}
is zero. The groups
$H^{i,j}(\xar,\mathbb{Z}/2)$ are known for $i\leq j\leq w-1$,
because of the following comparison result (cf Voevodsky
\cite[Corollary~6.9]{on2tors}):
\begin{theorem}
Assume that the Beilinson--Lichtenbaum Conjecture holds in degree
$n$. Then for any field $k$ and any smooth simplicial scheme
$\mathcal{X}$ over $k$,
\begin{enumerate}
\item the homomorphisms
$$H^{i,j}(\mathcal{X},\mathbb{Z}_{(p)})\to
  \mathbb{H}^{i}_{\et}(\mathcal{X},\mathbb{Z}_{(p)}(j))$$
are isomorphisms for $i-1\leq j\leq n$ and monomorphisms for $i=j+2$
  and $j\leq n$; and
\item the homomorphisms
$$H^{i,j}(\mathcal{X},\mathbb{Z}/p^m)\to
  \mathbb{H}^{i}_{\et}(\mathcal{X},\mathbb{Z}/p^m(j))$$
are isomorphisms for $i\leq j\leq n$ and monomorphisms for $i=j+1$ and
  $j\leq n$ and for any nonnegative integer $m$.
\end{enumerate}
\end{theorem}
In our case we have 
$$
\tilde{H}^{i,j}(\xar,\mathbb{Z}/2)\cong
\tilde{\mathbb{H}}_{\et}^i,(\xar, \mathbb{Z}/2(j))\cong 
\tilde{\mathbb{H}}_{\et}^i,(\Spec k, \mathbb{Z}/2(j))=0
$$
because of the inductive assumption on the
Beilinson--Lichtenbaum Conjecture holding through degree $w-1$ and
\cite[Lemma~7.3]{on2tors}. In other degrees almost nothing is known, 
except that 
$H^{2^w-1,2^{w-1}}(\xa,\mathbb{Z})=\tilde{H}^{2^w,2^{w-1}}
(\xar,\mathbb{Z})=0$ by Theorem 4.9, which uses the second result of
Rost \cite{rost-spinor}. Let $u$ be a nonzero class in the 
image of \eqref{mod2}, by means
of some hypothetical motivic cohomology operation $\theta$ acting on
$H^{*,*}(-,\mathbb{Z}/2)$, we can first try to move $u$ up to
the degree $(2^w,2^{w-1})$ and then compare it with the datum 
$\smash{\tilde{H}^{2^w,2^{w-1}}(\xar,\mathbb{Z})=0}$. What we are about to
write now is strictly related to \fullref{operaz} of this
manuscript. Let $\smash{Q_i^{\top}}$ be the
topological Steenrod operations defined by Milnor
\cite{milnor}. Let us assume that there exist operations which we
still denote by $Q_i$ that
act on $H^{*,*}(-,\mathbb{Z}/p)$ and that satisfy similar
properties as $\smash{Q_i^{\top}}$. In particular, we should expect that
$Q_i^2=0$ and we could compute the bidegrees of $Q_i$ from the equality
\begin{equation}\label{sba}
\smash{Q_i=Q_0(0,\ldots,1)+(0,\ldots,1)Q_0}
\end{equation}
where $\smash{(0,\ldots,1)}$, with the 1 in the $i$th place, is the
hypothetical motivical
cohomological operations defined as the dual to the canonical class
$\xi_i$ of the dual of the motivic Steenrod algebra
(cf Voevodsky \cite{voe-oper} and Milnor \cite{milnor}). Indeed,
the bidegree of $(0,\ldots,1)$ is $(2p^i-2,p^i-1)$ and
$Q_0$ should be the Bockstein, hence bidegree $(0,1)$. This shows that 
$|Q_i|=(2p^i-1, p^i-1)$ or $(2^{i+1}-1, 2^i-1)$ if $p=2$. 
\fullref{operaz} will be devoted to the construction of such
cohomology operations. Thus, $Q_{w-2}Q_{w-3}\cdots Q_1u$ belongs to 
$\smash{\tilde{H}^{2^w,2^{w-1}}(\xar,\mathbb{Z}/2)}$. By assumption, the class
$u$ is the reduction of an integral cohomology class and, by the
equality \eqref{sba}, so does the class $Q_{w-2}Q_{w-3}\cdots Q_1u$,
that therefore must be zero. To finish the proof it suffices to show
that multiplication by $Q_i$ on the relevant motivic cohomology group
$\tilde{H}^{w-i+2^{i+1}-1,w-i+2^1-1}(\xar,\mathbb{Z}/2)$ is
injective. Since
\begin{equation}
Q_i\cdot\tilde{H}^{w-i,w-1}(\xar,\mathbb{Z}/2)\subset
\tilde{H}^{w-i+2^{i+1}-1,w-i+2^1-1}(\xar,\mathbb{Z}/2)
\end{equation}
and the former
group is zero as mentioned above, the obstruction to injectivity of left
multiplication by $Q_i$ in degree $(w{-}i{+}2^{i+1}{-}1,w{-}i{+}2^1{-}1)$
is given
by the {\it $i$th Margolis homology} of the module 
$\tilde{H}^{*,*}(\xar,\mathbb{Z}/2)$ in degree
$(w{-}i{+}2^{i+1}{-}1,w{-}i{+}2^1{-}1)$. Recall that
given a graded left module over the Steenrod algebra $M$, the Margolis
homology of $M$ is defined as $H_1$ of the chain complex
$$\{M\stackrel{Q_i\cdot}{\longrightarrow}
  M\stackrel{Q_i\cdot}{\longrightarrow}M\}$$
and is usually denoted with $HM(M,Q_i)$.
\begin{remark}\label{vaff}
This argument, along with 
the successful use of
$Q_i$ to prove the vanishing of Margolis homology of $\xar$,
constituted the strongest motivation to construct cohomological
operations analogous to the Steenrod operations. Voevodsky did
this in \cite{voe-oper} and it turned out that 
the Hopf algebra (actually Hopf algebroid) structure of these 
operations may depend on the base field. More precisely, at odd
primes or if $\sqrt{-1}\in k$, then the multiplication and
comultiplication between the motivic Steenrod operations are as
expected, but at the prime $2$ and in the case $\sqrt{-1}\not\in
k$ both the product {\it and} the coproduct are different. For
more details see \fullref{duale}, \fullref{cartan},
\fullref{qt} of this text, \cite{voe-oper} and
\cite{io-mor}. From these formulae we see that the formula
\eqref{sba} barely holds even if $\sqrt{-1}\not\in k$.
\end{remark}  
It suffices to prove that $HM(H^{*,*}(\xar,\mathbb{Z}/2),Q_n)=0$ for
all $n\geq 0$. If $X$ is a smooth, projective variety of pure
dimension $i$, then $s_{i}(X)$ is the zero
cycle in $X$ represented by
$-\Newt_i(c_1(T_X),c_2(T_X),\ldots,c_i(T_X))$, that is the $i$th
Newton polynomial in the Chern classes of the tangent bundle of
$X$. Such polynomials depend only on the index $i$, thus the notation 
$s_{i}(X)$ for them may seem
strange. However, it is motivated by the property of these
polynomials being dual to certain monomials in the motivic homology of
the {\it classifying sheaf} $BU$. Indeed, in analogy to the
topological case, the set
$\{s_\alpha(\gamma)\}_{\alpha=(\alpha_1,\alpha_2,\ldots,\alpha_n,\ldots)}$ 
forms a basis for the motivic cohomology of $BU$, where $\alpha_i$ are
nonnegative integers and $\gamma\to BU$ is the colimit of universal
$n$--plane bundles over the infinite grassmanians
$Gr_n(\mathbb{A}^{\infty})$ (see the author's article \cite{io-mor}).  

\begin{theorem}\label{nullo}
Let $Y$ be a smooth, projective variety over $k$ with a map $X\to Y$
such that $X$ is smooth, projective variety of dimension $d=p^t-1$ with 
$\deg(s_{d}(X))\not\equiv 0\mod
p^2$. Then $HM(H^{*,*}(\tilde{C}(Y),\p),Q_t)=0$ for all $n\geq 0$.
\end{theorem}

To prove \fullref{zero} we apply the proposition to $X=X_i$
for smooth subquadrics $X_i$ of dimension $2^i-1$ for all $1\leq
i\leq w-2$.  The fact that the condition on the characteristic
number of the $X_i$ is satisfied is readily checked (see Voevodsky
\cite[Proposition~3.4]{on2tors}).

\begin{proof}
The pointed object $\tilde{C}(Y)$, which in topology
would just be equivalent to a point, it looks very much like $\Spec k$, 
but it is different enough to contain homotopy theoretic invariants 
of algebraic varieties mapping to $Y$. To prove the vanishing of the Margolis
homology, 
we analyze the two possible cases: in the first the scheme $Y$ has
points of degree prime with $p$, in the second the prime $p$ divides
all the degrees of points of $Y$. The former case implies that 
$\tilde{H}^{*,*}(\tilde{C}(Y),\p)=0$ by a transfer argument. Thus, we
can assume the latter property holds for $Y$. We want to present a 
contracting homotopy 
\begin{equation}
\phi\co \tilde{H}^{*,*}(\tilde{C}(Y),\p)\to
\tilde{H}^{*-2d-1,*-d}(\tilde{C}(Y),\p)
\end{equation}
satisfying the relation 
\begin{equation}\label{contr}
\phi Q_t-Q_t\phi=c
\end{equation}
for some nonzero $c\in\p$. In topology, the assumption on the
characteristic number of $X$ 
and the fact that all the caracteristic
numbers of $Y$ (a subset of degrees of points of $Y$) are divisible 
by $p$ are equivalent to $Q_t\tau\neq 0$ in the cohomology of the 
cone of $f_X\co  S^{d+n}\to
\Th(\nu_X)$ obtained via the Thom--Pontryagin construction, where $\nu$ 
is the normal bundle to an embedding
$X\hookrightarrow \mathbb{R}^{d+n}$ such that $\nu$ has a complex 
structure and $\tau\in H^{n}(\Th(\nu_X),\p)$
 is the Thom class. Existence of an ``algebraic''
Thom--Pontryagin construction (see Voevodsky \cite[Theorem~2.11]{on2tors} or
Borghesi \cite[Section~2]{io-for-grado}) allows to conclude the same
statement in the category $\hkp$. We will now transform the 
equality $Q_t\tau\neq 0$ in
the equality \eqref{contr}. Consider the sequence 
\begin{equation}
(\pro)^{\wedge d+n}\stackrel{f_X}{\to} \Th(\nu_X)\to \Cof(f_X)
\stackrel{\delta}{\to} (\pro)^{\wedge d+n}\wedge S_s^1\to\cdots 
\end{equation}
and let $Q_t\tau=c\gamma\neq 0$, where $\gamma=\delta^*\iota$ and
$\iota$ is the canonical generator of
$$H^{2(d+n)+1,d+n}((\pro)^{\wedge d+n}\wedge S_s^1,\p).$$
Consider now the chain of morphisms
\begin{equation}
\label{biglia}
\bfig
\barrsquare/<-`<-`->`->/<1700,500>[
  \tilde{H}^{*+2n,*+n}(\Cof(f_X),\p)`
  \tilde{H}^{*+2n,*+n}((\pro)^{n+d}{\wedge}S_s^1{\wedge}\tilde{C}(Y),\p)`
  \tilde{H}^{*,*}(\tilde{C}(Y),\p)`
  \tilde{H}^{*-2d-1,*-d}(\tilde{C}(Y),\p);
  \delta^*{\wedge}\id`
  \tau{\wedge}-`
  \cong`
  \phi]
  \efig
\end{equation}
Notice that $\delta^*\wedge \id$ is an isomorphism. Indeed,
one thing that makes $\tilde{C}(Y)$ similar
to $\Spec k$ is that $\Th(\nu_X)\wedge\tilde{C}(Y)$ is simplicially weak
equivalent to $\Spec k$. We can see this by smashing the sequence
\eqref{vai} with $\nu_+$ and using that the
projection $(\wcheck{C}(Y)\times \nu)_+\to \nu_+$ is a simplicial weak
equivalence because of the assumption on existence of a morphism $X\to
Y$ and \cite[Lemma~9.2]{on2tors}. Let now $y\in \tilde{H}^{*,*}(
\tilde{C}(Y),\p)$ and consider the equality 
\begin{equation}\label{bello}
c\gamma\wedge y=Q_t\tau\wedge y=\tau\wedge Q_t(y)-Q_t(\tau\wedge y)
\end{equation}
with the last equality following from the motivic Steenrod
algebra coproduct structure (cf \fullref{qt}, Voevodsky
\cite[Proposition~13.4]{voe-oper} or Borghesi \cite[Corollary 5]{io-mor})
and that $H^{i,j}(X,\p)=0$ if $X$ is a smooth variety and $i>2j$, hence
$Q_i\tau=0$ for $i<t$ because of the motivic Thom isomorphism. Equality
\eqref{bello} is equality \eqref{contr} if we let $\phi$ as in diagram
\eqref{biglia} and we identify isomorphisms.
\end{proof}

This finishes the proof of \fullref{zero} and the motivational
part of this manuscript.
\end{proof}


\section{Foundations and the Dold--Kan theorem}\label{topo}

We will proceed now with the creation of the environment in which
algebraic varieties preserve various algebraic invariants and at the
same time we can employ the homotopy theoretic techniques which have
been used by Voevodsky to prove \fullref{zero}. 
\begin{definition}\label{repr}
A functor $F\co \mathcal{E}\to \mathcal{F}$, with every object of
$\mathcal{F}$ being a set, is {\em representable} in the
category $\mathcal{E}$ if there exists an object $K_F\in\mathcal{E}$
such that the bijection of sets 
\begin{equation}\label{0}
F(X)\cong \Hom_{\mathcal{E}}(X, K_F)
\end{equation}
holds and it is natural in $X$ for any $X\in\mathcal{E}$.
\end{definition} 
In our situation, $\mathcal{F}$ will be the category of abelian
groups, and the congruence \eqref{0} is to be understood as abelian
groups. In algebraic topology there are several ways to prove 
that singular cohomology is a
representable functor in the unstable homotopy category 
$\mathcal{H}$. This is the ordinary homotopy category associated to
the model structure on topological spaces in which fibrations are the
{\it Serre} fibrations, as opposed to {\it Hurewicz} fibrations.
 Here we will recall one way to show such representability which
 serves as a model for proving the same result in the algebraic
 context. Originally, singular cohomology has
been defined as the homology of a (cochain) complex of abelian groups, 
hence it is resonable to find some connection between the category of 
complexes of abelian groups and the underlying objects of
$\mathcal{H}$. The latter maybe seen as a localized category of the
category of simplicial sets, thus we may set as our starting point the 
\begin{theorem}[Dold--Kan]
\label{Dold--Kan}
Let $\mathcal{A}$ be an abelian category with enough injective
objects. Then there exists a pair of adjoint functors $N$ and $K$ ($N$
left adjoint to $K$) which induce and equivalence of categories 
\begin{equation}
\xymatrix{\Ch_{\geq 0}(\mathcal{A}) \ar@<.5ex>[r]^{K} & 
\Delta^{\op}(\mathcal{A})\ar@<.5ex>[l]^{N}}
\end{equation} 
in which simplicially homotopic morphisms correspond to chain
homotopic maps and viceversa.
\end{theorem}

Consider the case of $\mathcal{A}$ to be the category of abelian
groups. Then we have a chain of adjunctions 
\begin{equation}\label{agg}
\xymatrix{\Ch_{\geq 0}(\mathcal{A}) \ar@<.5ex>[r]^{K} & 
\Delta^{\op}(\mathcal{A})\ar@<.5ex>[l]^{N}\ar@<.5ex>[r]^{\text {forget}}
& \Delta^{\op}(\Sets)\ar@<.5ex>[l]^{\mathbb{Z}[-]}\ar@<.5ex>[r]^{\hskip
  15pt|-|}& \Top\ar@<.5ex>[l]^{\hskip 15pt \Sing(-)}}
\end{equation}
where
$\mathbb{Z}[\mathcal{X}]$ is the free simplicial abelian group  
generated by the simplicial set $\mathcal{X}$, $|\mathcal{X}|$ is the
topological realization of $\mathcal{X}$, $\Sing(X)$ is the simplicial
set $\Hom_{\Top}(\Delta_{\top}^*, X)$ and $\Delta_{\top}^*$ is the
cosimplicial topological space
$$\Bigl\{(t_0,t_1,\ldots, t_n)/0\leq t_i\leq 1, \sum t_i=1\Bigr\}$$
Each of the categories appearing in the diagram admits localizing
structures which are preserved by the functors. In $\Ch_{\geq
  0}(\mathcal{A})$ the class of morphisms that is being inverted is
the quasi isomorphisms, in the simplicial homotopy categories are the
simplicial homotopy equivalences. This lead us to the first
result:
\begin{proposition}\label{primo}
The chain of adjunctions of diagram \eqref{agg} induce  bijection of
sets in the relevant localized categories
\begin{equation}
\Hom_{\mathcal{D}_{\geq 0}(\Ab)}(N\mathbb{Z}[\mathcal{X}], D_*)\cong
\Hom_{\mathcal{H}(\Delta^{\op}(\Sets))} (\mathcal{X}, K(D_*))
\end{equation}
for any simplicial set $\mathcal{X}$ and complex of abelian groups 
$D_*$.
\end{proposition}

Since we will need to take ``shifted'' complexes we have to enlarge
the category $\Ch_{\geq 0}(\Ab)$: let $\Ch_+(\Ab)$ be the category of 
  chain (ie with differential of degree $-1$) complexes of 
abelian groups which are bounded below, that is
for any $C_*\in \Ch_+(\Ab)$, there exists an integer $n_C\in\mathbb{Z}$
such that $C_i=0$ for all $i<n_C$. We relate these two categories of
chain complexes by means of a pair of adjoint functors
\begin{equation}
\xymatrix{\Ch_{\geq 0}(\Ab)\ar@<.5ex>[r]^{\text{forget}} &\Ch_+(\Ab)
\ar@<.5ex>[l]^{\text{truncation}}}
\end{equation}
where the truncation is the functor which sends a complex 
\begin{equation}
\{\cdots \to C_i\stackrel{d_i}{\to} C_{i-1}\to\cdots\to C_1\stackrel{d_1}{\to}
C_0\stackrel{d_0}{\to}\cdots\}
\end{equation}
to the complex 
\begin{equation}
\{\cdots \to C_i\stackrel{d_i}{\to} C_{i-1}\to\cdots\to C_1\stackrel{d_1}
{\to}\ker d_0\}
\end{equation} 
Notice that this functor factors through the derived
categories. Given a complex $D_*$ and an integer $n$, let $D[n]_*$ be
the complex such that $D[n]_i=D_{i-n}$. \fullref{primo}
implies
\begin{proposition}\label{secondo}
The adjoint functors of diagram \eqref{agg} induce a bijection of sets  
\begin{equation}\label{uno}
\Hom_{\mathcal{D}_+(\Ab)}(N\mathbb{Z}[\mathcal{X}], D[n]_*)\cong
\Hom_{\mathcal{H}(\Delta^{\op}(\Sets))}(\mathcal{X}, K(\mathrm{trunc}(D[n]_*))
\end{equation}
for any integer $n$, simplicial set $\mathcal{X}$ and bounded below
complex of abelian groups $D_*$. 
\end{proposition}

We wish to understand better the morphisms in the derived category in
order to use the equality \eqref{uno}. Recall that we want to show
that the functor {\it simplicial cohomology} of topological spaces 
is representable in the unstable homotopy category
$\mathcal{H}$. Quillen showed in \cite{quil} that this
category is equivalent to $\mathcal{H}(\Delta^{\op}(\Sets))$, thus
representability of singular cohomology in $\mathcal{H}$
amounts to the group isomorphism 
\begin{equation}
H^n_{\sing}(X,A)\cong 
\Hom_{\mathcal{D}_+(\Ab)}(N\mathbb{Z}[\Sing(X)], D_A[n]_*)
\end{equation}
for some bounded below complex of abelian groups $D_A$ and 
any CW-complex $X$.
\begin{definition}
Given a topological space $X$ and an abelian group $A$ , we define:
\begin{enumerate}
\item the $n$th singular homology group $H^{\sing}_n(X,A)$ to be the
  $n$th homology of the chain complex $N\mathbb{Z}[\Sing(X)]$;
\item the $n$th singular cohomology group $H_{\sing}^n(X,A)$, to be
  the $n$th homology of the cochain complex
  $\Hom_{\Ab}(N\mathbb{Z}[\Sing(X)],A)$.
\end{enumerate}
\end{definition}
In conclusion, we are reduced to compare the homology of the cochain complex
$\{\Hom_{\Ab}(N\mathbb{Z}[\Sing(X)],A)\}$ with the group 
$\Hom_{\mathcal{D}_+(\Ab)}(N\mathbb{Z}[\Sing(X)], D_A[n]_*)$ for some
appropriate complex $D_A[n]_*$. Morphisms in the derived category are
described by this important result

\begin{theorem}\label{der}
Let $\mathcal{A}$ be an abelian category with enough projective
(respectively injective) objects. Moreover, let $C_*$ and $D_*$ be
chain complexes, 
$P_{*}\to C_*$ and $D_*\to I_{*}$ quasi isomorphisms, $P_*$ being
projective objects and $I_*$ injective objects for all $*$.
Then 
\begin{equation}
\begin{split}
\Hom_{\mathcal{D}_+(\mathcal{A})}(C,D[n])&=H^n(\Tot^*(\Hom_{\mathcal{A}}
(P,D)))\\
(\text{resp.}&=H^{-n}(\Tot^*(\Hom_{\mathcal{A}}(C,I))))
\end{split}
\end{equation}
where $\Tot^*$ is the total complex of the cochain bicomplex
$\Hom_{\mathcal{A}}(P_{*},D_*)$
(resp. $\Hom_{\mathcal{A}}(C_*,I^{-*})$) and $H^*$ denotes the homology
of the cochain complex.   
\end{theorem}

\fullref{der} is rather classical and we refer to
Weibel \cite[Section~10.7]{weibel} for its proof. It is based on the
fact that, via calculus of fractions, one shows that
$$\Hom_{\mathcal{D}_+(\mathcal{A})}(C_*,D[n]_*)=\lim_{B_*\stackrel{\text{q.iso}}
{\to} C_*}\Hom_{\mathcal{K}_+(\mathcal{A})}(B_*,D[n]_*)$$ 
where  $\mathcal{K}_+(\mathcal{A})$ is the localization of $\Ch_+(\mathcal{A})$
with respect to the chain equivalences of complexes and 
$\Hom_{\mathcal{K}_+(\mathcal{A})}(B_*,D[n]_*)$ is the quotient set of 
$$\Hom_{\Ch_+(\mathcal{A})}(B_*,D[n]_*)$$ modulo chain
equivalences. According to the definitions given, applying \fullref{der} to $C_*=N\mathbb{Z}[\Sing(X)]$ and $D_*=A[i]$ (here the
abelian group $A$ is seen as complex concentrated in degree zero), we 
conclude that 
\begin{equation}
\Hom_{\mathcal{D}_+(\Ab)}(N\mathbb{Z}[\Sing(X)], A[i])=H^i_{\sing}(X,A)
\end{equation}
and, in view of \fullref{secondo}, we conclude
\begin{theorem}\label{rappr}
The homotopy class of $K(\mathrm{trunc}\,A[i])$ represents singular 
cohomology in the
sense that, for any topological space $X$ and abelian group $A$,
\begin{equation}
\Hom_{\mathcal{H}(\Delta^{\op}\Sets)}(\Sing(X),K(\mathrm{trunc}\,A[i]))=
H^i_{\sing}(X,A)
\end{equation}
or equivalently,
\begin{equation}
\Hom_{\mathcal{H}}(X,|K(\mathrm{trunc}\,A[i])|)=H^i_{\sing}(X,A)
\end{equation}
\end{theorem}
This represents the reasoning we wish to reproduce in the algebraic
setting.

 
\section{Motivic cohomology and its representability}\label{puggia}

\subsection{Sites and sheaves}
\label{sub:sites-sheaves}

There are two main issues we should consider: the first is finding the
category playing the role of $\mathcal{D}_+(\Ab)$ and the second is the
category replacing $\mathcal{H}$. The former question is dictated
by the definition we wish to give to {\it motivic cohomology groups}
$H^{i,j}(X, A)$ with coefficients in an abelian group $A$ of a smooth 
variety over a field $k$. If we set them to be right
hyperderived functors of $\Hom_{\mathcal{A}}(-,D_*)$ for some object
$D_*\in \Ch_+(\mathcal{A})$ and some abelian category $\mathcal{A}$
with enough projective or injective objects,
then \fullref{der} automatically tells us that motivic cohomology
is representable in the derived category $\mathcal{D}_+(\mathcal{A})$
and we are well poised for attempting to repeat the topological
argument of the previous section, with $\mathcal{D}_+(\mathcal{A})$ in
place of $\mathcal{D}_+(\Ab)$. Up to a decade or so ago the best
approximation of what we wished to be motivic cohomology was given by the
Bloch's {\it higher Chow groups} \cite{bloch} defined as the homology
of a certain complex. Most of Voevodsky's work in \cite{libro} is
devoted to prove that, if $k$ admits resolution of singularities,
for a smooth variety $X$, these groups are canonically isomorphic to 
$\mathbb{H}_{\Nis}^i(X,A(j))$: the 
Nisnevich (or even Zariski) hypercohomology of $X$ with coefficients 
in a certain complex of
sheaves of abelian groups $A(j)$. This makes possible to apply 
the previous remark since hypercohomology with coefficients in a
complex of sheaves $D_*$ is defined as right derived
functors of the global sections and
$\Hom_{\Ab-\Shv(\Sm/k)}(\mathbb{Z}\Hom_{\Sm/k}(-, X), D_*)$ are precisely
the global sections of $D_*$, by the Yoneda Lemma. Thus, we can take the
abelian category $\mathcal{A}$ to be $\mathcal{N}_k$, the
category of sheaves of abelian groups for the Nisnevich topology over
the site of smooth algebraic varieties over $k$, and $D_*$ to be
$A(j)$. Therefore, we define

\begin{definition}\label{motcoho}
The $(i,j)$th motivic cohomology group of a smooth algebraic variety 
$X$ is $\mathbb{H}_{\Nis}^i(X,A(j))$, where the complex of sheaves of
abelian groups $A(j)$ will be defined below. Such groups will be
denoted in short by $H^{i,j}(X,A)$.
\end{definition}

The quest for the algebraic counterpart of the category
$\mathcal{H}$ begins with the question on what are the ``topological
spaces'' or more precisely what should we take as
the category $\Delta^{\op}(\Sets)$ in the algebraic setting.
The first candidate is clearly $\Sch/k$, the category of schemes of
finite type over a field $k$. This category has all (finite) limits
but, the finiteness type condition on the objects of $\Sch/k$
prevents this category from being closed under (finite) colimits and 
this is too strong of a limitation on a category for trying to do 
homotopy theory on. The less painful way to solve this problem it
seems to be to
embed $\Sch/k$ in $\Funct((\Sch/k)^{\op}, \Sets)$, the category of
contravariant functors from  $\Sch/k$ to $\Sets$, via the Yoneda 
embedding $\mathbf{Y}\co X\to\Hom_{\Sch/k}(-, X)$. The category  
$\Funct((\Sch/k)^{\op}, \Sets)$ has all limits and colimits induced by
the ones of $\Sets$ and the functor $\mathbf{Y}$ has some good properties like
being faithfully full, because of Yoneda Lemma. However, although it
preserves limits, it does not preserve existing colimits in
$\Sch/k$. Thus, the entire category $\Funct((\Sch/k)^{\op}, \Sets)$ does
not reflect enough existing structures of the original category
$\Sch/k$. It turns out that smaller categories are more suited
for this purpose and the question of which colimits we wish to be
preserved by $\mathbf{Y}$ is related to {\it sheaf theory}. 
Let us recall some basic definitions of this theory.
\begin{definition}
Let $\mathcal{C}$ be a category such that for 
each object $U$ there exists a set of maps $\{U_i\to U\}_{i\in I}$, 
called a {\em covering}, satisfying the following axioms: 
\begin{enumerate}
\item for any $U\in\mathcal{C}$, $\{U\stackrel{\id}{\to} U\}$ is a
  covering of $U$; 
\item for any covering $\{U_i\to U\}_{i\in I}$ and any morphism
  $V\to U$ in $\mathcal{C}$, the fibre products $U_i\times_U V$ exist
  and $\{U_i\times_U V\to V\}_{i\in I}$ is a covering of $V$;
\item if $\{U_i\to U\}_{i\in I}$ is a covering of $U$, and if for each
  $i$, $\{V_{ij}\to U_i\}_{j\in J_i}$ is a covering of $U_i$, then the
  family $\{V_{ij}\to U\}_{i,j}$ is a covering of $U$.
\end{enumerate}
Then the datum $T$ of the covering is called a {\em Grothendieck topology}
and  the pair $(\mathcal{C}, T)$ is a {\em site}.
\end{definition}

\begin{examples}\label{esempi}
\begin{enumerate}
\item If $X$ is a topological space, $\mathcal{C}$ is the category
  whose objects are the open subsets of $X$, and morphisms are the
  inclusions, then the family $\{U_i\to U\}_{i\in I}$ for which
  $\amalg_iU_i\to U$ is surjective and $U$ has the quotient topology is an open
  covering of $U$. This is a Grothendieck topology on $\mathcal{C}$.
  In particular, if the topology of $X$ is the one of
  Zariski, we denote by  $X_{\Zar}$ the site associated to it.
\item Let $X$ be a scheme and $\mathcal{C}$ be the category whose
  objects are \'etale (ie locally of finite type, flat and 
  unramified) morphisms $U\to X$
  and as arrows the morphisms over $X$. If we let the covering be the
  surjective families of \'etale morphisms $\{U_i\to U\}_i$, then we
  obtain a site denoted with $X_{\et}$.
\item If $X$ is a scheme over a field $k$, and in the previous example
  we replace \'etale coverings with \'etale coverings 
  $\{U_i\to U\}_{i\in I-finite}$ with the property that for 
  each (not necessarily 
  closed point) $u\in U$ there is an $i$ and $u_i\in U_i$ such that
  the induced map on the residue fields $k(u_i)\to k(u)$ is an
  isomorphism, then we get a Grothendieck topology called {\em completely
  decomposed} or {\em Nisnevich}.
\item Let $\mathcal{C}$ be a subcategory of the schemes over a base
  scheme $S$ (eg $\Sm/k$: locally of finite type, smooth schemes over 
  a field $k$). Then Grothendieck
  topologies are obtained by considering coverings of objects of
  $\mathcal{C}$ given by surjective families of open embeddings,
  locally of finite type \'etale
  morphisms, locally of finite type \'etale morphisms with the condition 
  of the previous
  example, locally of finite type flat morphisms, et cetera.
The corresponding
  sites are denoted by $\mathcal{C}_{\Zar}$, $\mathcal{C}_{\et}$, 
$\mathcal{C}_{\Nis}$, $\mathcal{C}_{\fl}$. In the remaining part of this
  manuscript we will mainly deal with the site $(\Sm/k)_{\Nis}$.  
\end{enumerate}
\end{examples}
For the time being all the maps of a covering will be locally of
finite type. We have the following inclusions: 

$$
\hskip-10pt\{\text{Zariski~covers}\}\subset\{\text{Nisnevich~
  covers}\}\subset\{\text{\'etale~covers}\}\subset\{\text{flat~covers}\}
$$
We now return to the colimit preserving properties of the Yoneda
embedding functor $\mathbf{Y}\co \Sch/k\hookrightarrow \Funct((\Sch/k)^{\op}, \Sets)$. 
Taking for granted the notion of limits and
colimits of diagrams in the category $\Sets$, we recall that an 
object $A$ of a small category
$\mathcal{C}$ is a colimit of a diagram $D$ in $\mathcal{C}$ if
$\Hom_{\mathcal{C}}(A, X)=\lim_{\Sets} \Hom_{\mathcal{C}}(D, X)$ for all
$X\in\mathcal{X}$. The
diagrams in $\Sch/k$ whose colimits we are mostly interested in being 
preserved by $\mathbf{Y}$ are 
\begin{equation}\label{dia}
\begin{split}
\xymatrix{V\times_X V\ar[d]\ar[r] & V\\ V &}
\end{split}
\end{equation}
where $V\to X$ is a flat covering for a Grothendieck topology. Such
diagrams admit colimits on $\Sch/k$: they are each isomorphic to $X$. 
This can be
rephrased by saying that the square obtained by adding $X$ and the
canonical maps $V\to X$ on the lower right corner of diagram
\eqref{dia} is cocartesian.
In order for $\mathbf{Y}$ to
preserve such colimits we are going to consider the largest full
subcategories of $\Funct((\Sch/k)^{\op}, \Sets)$ in which $\mathbf{Y}(X)$ is 
the colimit of 
\begin{equation}
\begin{split}
\xymatrix{\mathbf{Y}(V\times_XV)\ar[r]\ar[d] & \mathbf{Y}(V)\\ 
\mathbf{Y}(V) &}
\end{split}
\end{equation}

\begin{definition}
Let $(\Sch/k, T)$ be a site. Then $F$ is a {\em sheaf} on the site
$(\Sch/k, T)$ if 
\begin{equation}
\begin{split}
\Hom_{\Funct((\Sch/k)^{\op},
  \Sets)}\left(\begin{matrix}
\xymatrix{\mathbf{Y}(V\times_XV)\ar[r]\ar[d]^{\mathbf{Y}(T)}  
& \mathbf{Y}(V)\ar[d]\\ \mathbf{Y}(V)\ar[r]^{\mathbf{Y}(T)} 
& \mathbf{Y}(X)}
\end{matrix}
, F\right)
\end{split}
\end{equation}
is a cartesian square in $\Sets$. This property will be
refered to as $F$ making the diagram 

$$\xymatrix{\mathbf{Y}(V\times_XV)\ar[r]\ar[d]^{\mathbf{Y}(T)}  
& \mathbf{Y}(V)\ar[d]\\ \mathbf{Y}(V)\ar[r]^{\mathbf{Y}(T)} 
& \mathbf{Y}(X)}$$
cocartesian in $\Funct((\Sch/k)^{\op}, \Sets)$.
\end{definition}

The full subcategory of $\Funct((\Sch/k)^{\op}, \Sets)$ of sheaves for the
$T$ topology will be denoted as $\Shv(\Sch/k)_T$. 

\begin{remark}
\begin{enumerate}
\item This definition is a rephrasing  of the classical one found in
  algebraic geometry texts. 
\item The Yoneda embedding $\mathbf{Y}$ embeds $\Sch/k$ as a full
  subcategory of $\Shv(\Sch/k)_T$ for any topology $T$ coarser than the
  flat.
\end{enumerate}
\end{remark}

Strange it may seem, the category $\Shv(\Sch/k)_T$ is going to be 
the replacement 
for the category $\Sets$ in topology. This time, though, the category depends 
upon one parameter: the Grothendieck topology $T$. The role played by
such parameter can be deduced by the results below. For the time being,
we will just consider sites with $\Sm/k$ as underlying category of
objects 

\begin{proposition}[Morel--Voevodsky {{\cite[Proposition~1.4, p96]{morvoe}}}]
$F$ is a sheaf for the Nisnevich topology if and only if $F$ makes
  cocartesian $\mathbf{Y}(D)$ where $D$ is the following diagram:
\begin{equation}\label{elementare}
\begin{split}
\xymatrix{U\times_XV\ar[r]\ar[d] &V\ar[d]^p \\
U\ar@{^(->}[r]^i & X}
\end{split}
\end{equation}
where all the schemes are smooth, $p$ is \'etale, $i$ is an open
immersion and
$$p\co p^{-1}(X-i(U))_{\red}\to (X-i(U))_{\red}$$
(ie reduced structure on the closed subschemes) is an isomorphism. Such
squares are called {\em elementary squares}.
\end{proposition}

\begin{remark}
There are analogous results for the Zariski and \'etale topologies
  according to whether we require $p$ to be an open embedding or $i$
  to be just \'etale, respectively.
\end{remark}
We now drop the $\mathbf{Y}$ from the notation and therefore think of
a scheme $X$ as being $\mathbf{Y}(X)$, the sheaf represented by $X$, as an
object in $\Shv(\Sch/k)_T$. Denote with $\Cof(f)$ the colimit (of course in
the category $\Shv(\Sch/k)_T$) of the diagram
$$
\xymatrix{U\ar[r]^f\ar[d] & X\\
\Spec k & }
$$
and call it the {\em cofiber} of the map $f$. If $f$ is injective, we
will write $X/U$ for such cofiber. By general nonsense, in
a cocartesian square we can identify cofibers of parallel maps. In
particular, in the case of the cocartesian square \eqref{elementare}
we get a canonical isomorphism $V/(U\times_XV)\to X/U$. This implies
that we are able to identify special sheaves in the isomorphism
classes of such objects.
\begin{examples}
\begin{enumerate}
\item Let $Z_1$ and $Z_2$ be two disjoint closed subvarieties of a smooth
  variety $X$. Then,
\begin{equation}
\begin{split}
\xymatrix{X-Z_1-Z_2\ar[r]\ar[d] & X-Z_1\ar[d]\\
X-Z_2\ar[r] & X}
\end{split}
\end{equation}
is an elementary square and the canonical morphism
$$X-Z_1/(X-Z_1-Z_2)\to X/(X-Z_2)$$ is, not surprisingly, an
isomorphism.  
\item Let $L/k$ be a finite, separable field extension  $y\co \Spec L\to
  X$ be an $L$--point, $X$ a smooth variety over $k$ and $X_L$ the base 
change of $X$ over $L$. Then there exist points
  $x_1,x_2\cdots,x_n,x_L$ of $X_L$ with $x_L$ rational (if $L/k$ is
  a Galois extension then all $x_i$ are rational), such that 
\begin{equation}
\begin{split}
\xymatrix{X_L-\{x_1,\ldots,x_n,x_L\} \ar[r]\ar[d] &
  X_L-\{x_1,\ldots,x_n\}\ar[d]^p\\ X-y\ar[r]^i & X}
\end{split}
\end{equation}
is an elementary square. For instance, 
\begin{equation}
\begin{split}
\xymatrix{\mathbb{A}_{\mathbb{C}}^1-\{i,-i\}\ar[r]\ar[d] &
  \mathbb{A}_{\mathbb{C}}^1-\{i\}\ar[d]\\
\mathbb{A}^1_{\mathbb{R}}-\Spec \mathbb{C}\ar[r] &
  \mathbb{A}_{\mathbb{R}}^1} 
\end{split}
\end{equation}
is an elementary square and
$\mathbb{A}_{\mathbb{C}}^1-\{i\}/(\mathbb{A}_{\mathbb{C}}^1-\{i,-i\})\to 
\mathbb{A}_{\mathbb{R}}^1/(\mathbb{A}^1_{\mathbb{R}}-\Spec \mathbb{C})$
is an isomorphism.      
\end{enumerate}
\end{examples}

Combining the two previous examples we prove:
\begin{proposition}\label{pu}
Let $L/k$ be a finite, separable field extension. Assume that a smooth
variety has an $L$--point $y\co \Spec L\to X$. Then, the canonical morphism
$X/(X-y)\to X_L/(X_L-y)$ is an isomorphism of sheaves for a topology
finer or equivalent to the Nisnevich one. 
\end{proposition}

The next results require the introduction of some notation most of the
homotopy theorists are familiar with.
\begin{definition}
\begin{enumerate}
\item A {\em pointed sheaf} is a pair $(F,x)$ of a sheaf $F$ and an
  element $x\in F(\Spec k)$. The symbol $F_+$ will denote the sheaf
  $F\amalg \Spec k$. The cofiber of the unique morphism $\emptyset\to
  X$ is set by convention to be $X_+$. Notice that, 
unlike in topology, there exist sheaves $F$ with the property that
 $F(\Spec k)$ is 
the empty set, thus such $F$ cannot be pointed (unless by adding a
  separate basepoint). This fact is the source of the most interesting
  phenomena in motivic homotopy theory such as the various degree 
  formulae.
\item Let $(A,a)$ and $(B,b)$ two pointed sheaves. Then $A\vee B$ is
  the colimit of 
\begin{equation}
\begin{split}
\xymatrix{\Spec k\ar[r]^a\ar[d]^b & A\\ B &}
\end{split}
\end{equation}
pointed by $a=b$.
\item The pointed sheaf $A\wedge B$ is defined to be $A\times B/A\vee
  B$ pointed by the image of $A\vee B$.
\item Let $V\to X$ be a vector bundle (ie the scheme associated with
  a locally free sheaf of
  $\mathcal{O}_X$ modules). Then the {\em Thom sheaf} of $V$ is
  $V/(V-i(X))$, where $i\co X\to V$ is the zero section of $V$. The Thom
  sheaf will be denoted by $\Th_X(V)$.
\item The sheaf $T$ will always stand for $\aff/(\aff-0)$.
\end{enumerate}
\end{definition} 
The Thom sheaves enjoy of properties similar to the topological
counterparts. 

\begin{proposition}\label{ma}
Let $V_1\to X_1$ and $V_2\to X_2$ be two vector bundles. Then there is
a canonical isomorphism of pointed sheaves 
$$\Th_{X_1\times X_2}(V_1\times V_2)\stackrel{\cong}{\to}
\Th_{X_1}(V_1)\wedge Th_{X_2}(V_2)$$
\end{proposition} 

\begin{corollary}\label{cheneso}
Let $\mathbb{A}^n_X:=\mathbb{A}^n_k\times_k X$ be the trivial vector 
bundle of rank $n$ over
$X$. Then there is a canonical isomorphism of pointed sheaves 
$$\Th(\mathbb{A}^n_X)\stackrel{\cong}{\to} T^n\wedge X_+$$
\end{corollary}
\begin{proof}
\fullref{ma} implies that $\Th(\mathbb{A}^n_k)\cong
T^{\wedge n}$ and 
$$\Th_X(\mathbb{A}^n_X)=Th_{\Spec k\times_k
  X}(\mathbb{A}^n_k\times_k\mathbb{A}^0_X)\cong T^{\wedge n}\wedge
\Th_X(\mathbb{A}^0_X)=T^{\wedge n}\wedge X_+\proved$$
\end{proof}

\begin{corollary}\label{purezza1}
Let $L/k$ be a finite and separable field extension and
$y\co \Spec L\to\mathbb{A}_k^n$ be an $L$ point. For any positive
integer $n$ there is an isomorphism of sheaves 
$$\mathbb{A}_k^n/(\mathbb{A}_k^n-y)\cong T^{\wedge n}\wedge
(\Spec L)_+$$ 
if the topology $T$ is finer or equivalent to the Nisnevich one.
\end{corollary}
 
\begin{proof}
\fullref{pu} gives an isomorphism 
$\mathbb{A}_k^n/(\mathbb{A}_k^n-y)\cong
\mathbb{A}_L^n/(\mathbb{A}_L^n-y)$. Since $\mathbb{A}_L^n$ is the
trivial rank $n$ vector bundle over $\Spec L$, \fullref{cheneso}
gives
$$\mathbb{A}_L^n/(\mathbb{A}_L^n-y)\cong \Th(\mathbb{A}_L^n)\cong
T^{\wedge n}\wedge (\Spec L)_+.\proved$$
\end{proof}

This last isomorphism can be generalized to any closed embedding of
smooth schemes, if we allow to work in a suitable localized category
of $\Delta^{\op}\Shv(\Sm/k)_{\Nis}$. In that situation the equivalence will
be $M/(M-i(X))\cong \Th_X(\nu_i)$ if $i\co X\hookrightarrow M$ is a closed
embedding of smooth schemes. We will go back to this in the next
section.


\subsection{Homotopy categories of schemes}

In \fullref{sub:sites-sheaves} we have mentioned that we will be taking the
category $\Delta^{\op}\Shv(\Sm/k)_{\Nis}$ as replacement for
$\Delta^{\op}(\Sets)$ in the program of \fullref{topo}. The chain
of adjunctions of diagram \eqref{agg} has this algebraic counterpart:
\begin{equation}\label{agg-alg}
\xymatrix{\Ch_+(\mathcal{N}_k)\ar@<.5ex>[r]^{\text{trunc}} &
\Ch_{\geq 0}(\mathcal{N}_k) \ar@<.5ex>[r]^{K}\ar@<.5ex>[l]^{\text{forg}} & 
\Delta^{\op}(\mathcal{N}_k)\ar@<.5ex>[l]^{N}\ar@<.5ex>[r]^{\hskip -25pt
\text {forget}}
& \Delta^{\op}\Shv(\Sm/k)_{\Nis}\ar@<.5ex>[l]^{\hskip -25pt\mathbb{Z}[-]}}
\end{equation}
where $\mathcal{N}_k$ denotes the category of sheaves of abelian
groups for the Nisnevich topology (the definition is the same as for
$\Shv(\Sm/k)_{\Nis}$ except that it is a full subcategory of
$\Funct((\Sm/k)^{\op},\Ab)$). Notice that the diagram \eqref{agg-alg} is a
generalization of \eqref{agg} since there is a full embedding
$\Sets\hookrightarrow \Shv(\Sm/k)_T$ sending a set $S$ to the 
sheaf associated to the presheaf such that $F(U)=S$ for any $U\in \Sm/k$ and 
$F(f)=\id$ for any $f$ morphism
in $\Sm/k$. In \fullref{topo} we passed to localized categories
associated with the categories appearing in the adjunction. The
functors respected the localizing structures and induced
adjunctions on the localized categories. Here we will do the
same. Since $\mathcal{N}_k$ is an abelian category, it admits a derived
category $\mathcal{D}(\mathcal{N}_k)$. In general, localizing a
category may be a very complicated operation to do on a category,
being unclear even what morphisms look like in the localized category,
assuming it exists. D Quillen in \cite{quil} developed a systematic
approach to this issue, under the assumption that the category in
question had a {\em model structure}. This is a set of axioms which 
three classes of morphisms (weak equivalences, cofibrations and
fibrations) have to satisfy. In
a sense, this makes the category similar to the category of topological
spaces and the localization with repect of weak equivalences becomes
treatable much in the same way as the unstable homotopy category of
topological spaces $\mathcal{H}$. The category
$\Delta^{\op}(\Shv(\Sm_k)_{T}$ can be given a structure of model category
as follows:
\begin{definition}\label{de}
\begin{enumerate}
\item A {\em point} in $\Delta^{\op}\Shv(\mathcal{C})_T$ is a functor
  $\Delta^{\op}\Shv(\mathcal{C})_T\to \Sets$ commuting with finite limits
  and all colimits.
\item\label{rlp} Let a square
$$\xymatrix{A\ar[r]\ar[d]_g & B\ar[d]_f \\
  C\ar[r]\ar@{.>}[ur]^{\exists h} & D}$$
  be given in a category. We say
  that $f$ has the {\em right lifting property} with respect of $g$ if
  there exists a lifting $h$ making the diagram commute. In the same
  instance, we say that $g$ has the {\em left lifting property} with
  respect of $f$.
\end{enumerate}
\end{definition}

\begin{definition}\label{modelli}
Let $\mathcal{C}_T$ be a site and $f\co \mathcal{X}\to\mathcal{Y}$ a 
morphism in the site $\Delta^{\op}\Shv(\mathcal{C}_T)$. Then $f$ is called
\begin{enumerate}
\item a {\em simplicial weak equivalence} if for any ``point'' $x$ of 
$\Delta^{\op}\Shv(\mathcal{C})_T$ the morphism of simplicial sets
  $x(f)\co x(\mathcal{X})\to x(\mathcal{Y})$ is a simplicial weak
  equivalence;
\item a {\em cofibration} if it is a monomorphism;
\item a {\em fibration} if it has the right lifting property with
  respect to any cofibration which is a weak equivalence.
\end{enumerate}
\end{definition}

\begin{examples}
 Consider the site $\mathcal{C}=(\Sch/k)_{\Zar}$ or $(\Sm/k)_{\Zar}$ and
  $x\co \Spec L\to X$ be a point of a scheme $X$ over $k$. Let $I_{x}$ be
  the cofiltered category of Zariski open subschemes of $X$ containing
 $x$. Then the
  functor $F\to\colim_{U\in I_x}F(U)$ is a point for
  $\Shv(\mathcal{C})_{\Zar}$. If we change the topology to Nisnevich or
 \'etale, to get points for the respective categories we can take
 $$I_x^{\Nis}=\{p\co U\to X\, \text{\'etale~such~that}~ \exists y\in p^{-1}x\,
\text{with}~ p\co k(y)\stackrel{\cong}{\to}k(x)\}$$ and
 $$I_x^{\et}=
\left\{\text{diagrams}
\begin{matrix}
\xymatrix{ & U\ar[d]^p\\
 \Spec K\ar[r]^{\bar{x}}
 & X}
\end{matrix}\,\text{with}\, p\; \text{\'etale}\right\}$$ 
where $\bar{x}\co \Spec K\to
 \Spec k(x)\stackrel{x}{\to} X$ and $K$ is a separably algebraically
 closed field. Notice that the latter is a $K$ dependent category.
\end{examples}

Jardine \cite{jard} checked that the classes of morphisms of \fullref{modelli} satisfy the axioms of a model structure. Its 
homotopy category, that is the localization inverting the weak equivalences,
will be denoted as $\mathcal{H}_s(k)$. A generalization of the
Dold--Kan Theorem implies

\begin{proposition}\label{secondo*}
The maps in diagram \eqref{agg-alg} preserve the mentioned localizing
structures. In particular, the adjoint functors in that diagram induce
a bijection of sets
\begin{equation}
\Hom_{\mathcal{D}_+(\mathcal{N}_k)}(N\mathbb{Z}[\mathcal{X}],
D[i]_*)\cong \Hom_{\mathcal{H}_s(k)}(\mathcal{X},
K(\mathrm{trunc}(D[i]_*)))
\end{equation}
for any $D_*\in \Ch_+(\mathcal{N}_k)$,
$\mathcal{X}\in\Delta^{\op}\Shv(\Sm/k)_{\Nis}$ and $i\in\mathbb{Z}$.
\end{proposition} 
Associating to any set $S$ the constant (simplicially and as a
sheaf) simplicial sheaf $S$, we can embed $\Delta^{\op}\Sets$ in
$\Delta^{\op}\Shv(\Sm/k)_{\Nis}$. Since this embedding preserves
the localizing structures we get a faithfully full embedding of
$\mathcal{H}$ in $\mathcal{H}_s(k)$. In particular, the canonical
projection of $\Delta^1\to \Spec k$ becomes an isomorphism in
$\mathcal{H}_s(k)$. Thus \fullref{secondo*}
is a generalization of \fullref{secondo}. However, while in
the topological case $\mathcal{H}$ is the correct category to study
topological spaces with respect of homotopy invariant functors like
singular cohomology, the category $\mathcal{H}_s(k)$ is not
appropriate for certain applications: for example \fullref{purezza} is false in $\mathcal{H}_s(k)$. This category
can be made more
effective to study algebraic varieties via motivic cohomology. Indeed,
we can localize further $\mathcal{H}_s(k)$ without loosing any
information detected by motivic cohomology: it is known that for any
algebraic variety $X$, the canonical projection $X\times_k\aff\to X$
induces via pullback of cycles an isomorphism on (higher) Chow groups
and the same holds for the definition of motivic cohomology of
\fullref{motcoho}, if $k$ is a perfect field. 
Therefore, inverting all such
projections is a lossless operation with respect to these functors.
To localize a localized category we can employ the Bousfield
framework \cite{bous} and define 
$\mathcal{X}\in\Delta^{\op}\Shv(\Sm/k)_{\Nis}$ to be a simplicial
sheaf $\aff$ {\it local} if for any
$\mathcal{Y}\in\Delta^{\op}\Shv(\Sm/k)_{\Nis}$, the maps of sets
$$\Hom_{\mathcal{H}_s(k)}(\mathcal{Y},\mathcal{X})\stackrel{p^*}{\to}
\Hom_{\mathcal{H}_s(k)}(\mathcal{Y}\times_k\aff,\mathcal{X})$$ induced by
the projection $\mathcal{Y}\times_k\aff\to\mathcal{Y}$, is a
bijection. 
\begin{definition}\label{bou}
A map $f\co \mathcal{X}\to\mathcal{Y}$ in $\Delta^{\op}\Shv(\Sm/k)_{\Nis}$ is
called:
\begin{enumerate}
\item an {\em $\aff$ weak equivalence} if for any $\aff$ local sheaf
  $\mathcal{Z}$, the map of sets 
$$\Hom_{\mathcal{H}_s(k)}(\mathcal{Y},\mathcal{Z})\stackrel{f^*}{\to}
\Hom_{\mathcal{H}_s(k)}(\mathcal{X},\mathcal{Z})$$
ia a bijection;
\item an $\aff$ cofibration if it is a monomorphism;
\item an $\aff$ fibration if it has the right lifting property
  (\fullref{de} \eqref{rlp}) with respect to monomorphisms that
  are $\aff$ weak equivalence.
\end{enumerate}
\end{definition}

Morel and Voevodsky showed in \cite{morvoe} that this is a (proper) model
structure on $\Delta^{\op}\Shv(\Sm/k)_{\Nis}$. The associated homotopy
category will be denoted as $\mathcal{H}(k)$ and called the {\it
  unstable homotopy category of schemes over $k$}.
\begin{remark}[see \cite{morvoe}]
The inclusion of the full subcategory
  $\smash{\mathcal{H}_{s,\aff}}(k)\hookrightarrow\mathcal{H}_s(k)$ of $\aff$
  local objects has a left adjoint, denoted by $\Sing(-)$, that identifies
  $\smash{\mathcal{H}_{s,\aff}}(k)$ with $\mathcal{H}(k)$. Such functor is
  equipped of a natural transformation $\Theta\co \Id\to \Sing$ that is an 
  $\aff$ weak equivalence on objects.
\end{remark}

Working in the category $\mathcal{H}(k)$ as opposed to
$\Delta^{\op}\Shv(\Sm/k)_{\Nis}$ makes schemes more ``flexible'' quite like
working with $\mathcal{H}$ in topology. In this case, however, we have
an high algebraic content, thus such flexibility should be made more
specific with a few examples. The next results are stated in unpointed
categories, although they could be formulated in the pointed setting
just as well.

\begin{proposition}
 Let $V\to X$ be a vector bundle and
  $\mathbb{P}V\hookrightarrow\mathbb{P}(V\oplus\affx)$ be the
  inclusion of the projectivized scheme induced by the zero section $$V\to
  V\oplus\mathbb{A}_X^1.$$ Then there exists a canonical morphism
  $\mathbb{P}(V\oplus\affx)/\mathbb{P}V\to \Th(V)$ in $\Shv(\Sm/k)_{\Nis}$ 
that is an $\aff$ equivalence.
\end{proposition}
\begin{proof}
See Morel--Voevodsky \cite[Proposition~2.17, p112]{morvoe}.
\end{proof}

\begin{corollary}
The canonical morphism $\mathbb{P}^n_k/\mathbb{P}^{n-1}_k\to T^{\wedge
  n}$ is an $\aff$ weak equivalence for any positive integer $n$.
\end{corollary}
One of the first consequences of the algebraic data lying in the category
$\mathcal{H}(k)$ is the existence of several nonisomorphic ``circles''
in that category. Let $S_s^1$ denote the sheaf
$\Delta^1/\partial\Delta^1$ and $\mathbb{G}_m$ be $\aff-\{0\}$.
The following result is a particularly curious one and states that the
simplicial sheaf $T=\aff/(\aff-\{0\})$, which should embody the
properties of a sphere or a circle, is made of an algebraic and a
purely simplicial part. 

\begin{proposition}
There is an $\aff$ weak equivalence $\mathbb{G}_m\wedge S_s^1\cong T$.
\end{proposition}

\begin{proof} See Morel--Voevodsky \cite[Lemma 2.15]{morvoe}.
\end{proof}

In the previous few results the Nisnevich topology did not play any
role. On the contrary, the theorem below uses this
assumption. Arguably, it is the most important result proved in
\cite{morvoe} and it represents a very useful tool which is crucial in
several applications such as the computation of the motivic Steenrod
operations (see Voevodsky \cite{voe-oper}) and in the proof of the
Milnor Conjecture (see Voevodsky \cite{on2tors}). At this point, it
appears to be the main reason to work in the category $\mathcal{H}(k)$
instead of $\mathcal{H}_s(k)$.

\begin{theorem}[Morel--Voevodsky {{\cite[Theorem~2.23, p115]{morvoe}}}]
\label{purezza}
Let $i\co Z\hookrightarrow X$ be a closed embedding of smooth schemes and
$\nu_i$ its normal bundle. Then there is a canonical isomorphism in
$\mathcal{H}(k)$
$$X/(X-i(Z))\cong \Th(\nu_i)$$
\end{theorem}

Notice the analogy between this congruence of Nisnevich sheaves and
the Nisnevich sheaf isomorphism of \fullref{purezza1}. This is
even more evident in the following corollary.

\begin{corollary}
Let $L/k$ be a finite and separable field extension and $X$ a smooth
scheme of dimension $n$ with an $L$ point $x\co \Spec L\to X$. 
Then there is an isomorphism
$$X/(X-x)\cong T^{\wedge n}\wedge (\Spec L)_+$$
in the category $\mathcal{H}(k)$.
\end{corollary}

Let us go back to the adjunction \eqref{agg-alg} and representability 
of motivic cohomology. From \fullref{secondo*} and
\fullref{der}, we see that for any scheme $X$ and complex
$D_*\in \Ch_+(\mathcal{N}_k)$
\begin{equation}\label{sgiang}
\mathbb{H}^{i}(X, D_*)\cong
\Hom_{\mathcal{D}_+(\mathcal{N}_k)}(\mathbb{Z}[X], D[i])
\cong \Hom_{\mathcal{H}_s(k)}(X, K(\text{trunc}\,
D[i]_*))
\end{equation}
To see this, recall that the hypercohomology groups of a scheme $X$ are
 the right 
hyperderived functors of the global sections functor (on the small
site in question). Since
$\Hom_{\Ab-\Shv}(\mathbb{Z}\Hom_{\Sch/k}(-,X), F)=F(X)$, we conclude that
hypercohomology groups are $Ext$ functors of
$\mathbb{Z}[X]:=\mathbb{Z}\Hom_{\Sch/k}(-,X)$. This is isomorphic to
the left end side group in the isomorphism of \fullref{secondo*}
because $\mathbb{Z}[X]$ is a complex concentrated in dimension zero
(hence $N\mathbb{Z}[X]=\mathbb{Z}[X]$)
and of \fullref{der}. In particular, for any abelian group $A$
we get that (cfr. \fullref{motcoho}) 
$H^{i,j}(X,A)\cong \Hom_{\mathcal{H}_s(k)}(X, K (A(j)[i]))$, that is
motivic cohomology is a representable functor in
$\mathcal{H}_s(k)$. We now wish to derive an adjunction allowing to
conclude an identification similar to the one of \fullref{secondo*} but involving the category $\mathcal{H}(k)$. 
Since things work fine in the case of $\mathcal{D}_+(\mathcal{N}_k)$
and $\mathcal{H}_s(k)$ we may try to alterate the former category in
the same way we already did to obtain $\mathcal{H}(k)$ from
$\mathcal{H}_s(k)$. This leads us to the notion of complex $D_*\in
\Ch_+(\mathcal{N}_k)$ to be $\aff$ {\it local} if the morphism
$\aff\stackrel{p}{\to} \Spec k$ induces an isomorphism of groups
$$p^*\co \Hom_{\mathcal{D}_+(\mathcal{N}_k)}(C_*,
D_*)\stackrel{\cong}{\to}
\Hom_{\mathcal{D}_+(\mathcal{N}_k)}(C_*\stackrel{L}{\otimes}\mathbb{Z}[\aff],
D_*)$$
for any $C_*\in \Ch_+(\mathcal{N}_k)$. We than can endow
$\mathcal{D}_+(\mathcal{N}_k)$ of a model structure given by the
classes of maps described in \fullref{bou} adapted to this
derived category. The resulting category
$\mathcal{D}_+^{\affs}(\mathcal{N}_k)$ is equivalent to the
full subcategory of $\mathcal{D}_+(\mathcal{N}_k)$ of $\aff$ local
objects, and the inclusion functor $i$ has a left adjoint denoted by
$\Sing(-)$. One checks that
the pair of functors induced by $(N\circ\mathbb{Z}[-],
K\circ\text{trunc})$ of diagram \eqref{agg-alg} on derived and
simplicial homotopy categories preserve $\aff$ local objects,
therefore they induce an adjunction between
$\mathcal{D}_+^{\affs}(\mathcal{N}_k)$ and $\mathcal{H}_s(k)$:
\begin{theorem}\label{finale}
The pair of adjoint functors $(N\circ\mathbb{Z}[-],
K\circ\mathrm{trunc})$ induce a bijection of sets:
\begin{equation}\label{put}
\Hom_{\mathcal{D}_+^{\affs}(\mathcal{N}_k)}(N\mathbb{Z}[\mathcal{X}],
D[i]_*)\cong \Hom_{\mathcal{H}(k)}(\mathcal{X},K(\text{trunc}\,D[i]_*))
\end{equation}
for any $D_*\in \Ch_+(\mathcal{N}_k)$,
$\mathcal{X}\in\Delta^{\op}\nshsets$ and $i\in\mathbb{Z}$.
\end{theorem}    
To relate the left end side of the isomorphism \eqref{put} with
hypercohomology groups and hence with motivic cohomology, we employ
the adjunction $(\Sing(-),i)$: since
$\mathcal{D}_+^{\affs}(\mathcal{N}_k)$ is equivalent to the full
subcategory of $\aff$ local objects of $\mathcal{D}_+(\mathcal{N}_k)$,
it follows that
$$\Hom_{\mathcal{D}_+^{\affs}(\mathcal{N}_k)}(C_*,D_*)=
\Hom_{\mathcal{D}_+(\mathcal{N}_k)}(C_*, D_*)$$ 
for any $\aff$ local complexes $C_*$ and $D_*$. 
Let $\Sing\co \mathcal{H}_s(k)\to\mathcal{H}_s^{\aff}(k)$ denote also the 
functor left adjoint to the inclusion of the full subcategory
generated by the $\aff$ local simplicial sheaves in $\mathcal{H}_s(k)$. 
Since $N\mathbb{Z}[-]$ preserves $\aff$ local objects and $\aff$ weak
equivalences, $N\mathbb{Z}[\Sing(X)]$ is an $\aff$ local complex and
the canonical morphism $N\mathbb{Z}[\Theta]\co N\mathbb{Z}[X]\to 
N\mathbb{Z}[\Sing(X)]$ is an $\aff$ weak equivalence. Therefore, 
\begin{multline*}
\Hom_{\mathcal{D}_+^{\affs}(\mathcal{N}_k)}(N\mathbb{Z}[\Sing(X)],
L_*)=\Hom_{\mathcal{D}_+(\mathcal{N}_k)}(N\mathbb{Z}[\Sing(X)],
L_*)\stackrel{N\mathbb{Z}[\Theta]^*}{\longrightarrow} \\
\Hom_{\mathcal{D}_+(\mathcal{N}_k)}
(N\mathbb{Z}[X], L_*)\cong\mathbb{H}^i_{\Nis}(X, L_*)
\end{multline*}
is a group isomorphism for any $\aff$ local complex $L_*$. 
In the case of motivic cohomology, if $k$ is a perfect field, the
complex of sheaves $A(j)$ are $\aff$ local for all $j\in\mathbb{Z}$, 
hence:

\begin{corollary}\label{finale*}
For a smooth scheme $X$ over a perfect field $k$, and an abelian
group $A$, we have 
\begin{equation}
\begin{aligned}
H^{i,j}(X,A) &\cong
\Hom_{\mathcal{D}_+^{\affs}(\mathcal{N}_k)}(\mathbb{Z}[\Sing(X)], A(j)[i]) \\
&\cong \Hom_{\mathcal{H}(k)}(X, K(\mathrm{trunc}\,A(j)[i]))
\end{aligned}
\end{equation}
\end{corollary}

The homotopy class of the simplicial sheaf $K(\mathrm{trunc}\,A(j)[i])$ is
called the $(i,j)$ {\it motivic Eilenberg--MacLane} simplicial sheaf with
coefficients in $A$, in analogy with the topological representing
space of singular cohomology and will be denoted by $K(A(j),i)$.

\section{Motivic cohomology operations}\label{operaz}

In this section the word {\it sheaf} will always mean {\it Nisnevich
  sheaf}, unless otherwise specified.
The objective of this section is to give a more solid basis to key
  techniques used to prove \fullref{nullo}, namely the use of
  {\it motivic cohomology operations} with finite coefficients.
 By such 
term we mean natural transformations
$$H^{*,*}(-,\p)\to H^{*+n,*+m}(-,\p)$$
commuting with suspension isomorphisms
$$\tilde{H}^{*,*}(\mathcal{X},\p)\stackrel{\sigma_s}{\to}
  \tilde{H}^{*+1,*}(\mathcal{X}\wedge S^1_s,\p)$$ and
$$\tilde{H}^{*,*}(\mathcal{X},\p)\stackrel{\sigma_t}{\to}\tilde{H}^{*+1,*+1}
(\mathcal{X}\wedge\mathbb{G}_m,\p)$$ 
for any pointed sheaf $\mathcal{X}$. See Voevodsky \cite[Theorem~2.4]{voe-oper}
for a proof that these canonical morphisms are 
isomorphisms in the case the base field is perfect. 
 We are interested to produce motivic cohomology
operations as similar as possible to the ones generating the 
{\it Steenrod algebra} $\mathcal{A}_{\top}^*$: they are classically 
denoted by $P^i$ and by $\beta$, the {\it Bockstein} operation, and 
the collection of them over all nonnegative $i$ 
generate $\mathcal{A}_{\top}^*$ as (graded) algebra over $\p$. At any
prime $p$, the operation $P^0$ is the identity. At the prime $p=2$,
the operation $P^i$ is usually denoted by $\Sq^{2i}$ and $\beta P^i$ by
$\Sq^{2i+1}$ and are sometimes called the {\it Steenrod squares}. This
is a classical subject and there are several texts available covering 
it. Among them, the ones of Steenrod and Epstein \cite{steen-eps} and
of Milnor \cite{milnor} contain the constructions and ideas which will 
serve as models to follow in the ``algebraic'' case. The topological
Steenrod algebra is in fact an Hopf algebra and is completely
determined by 
\begin{enumerate}
\item generators $\{\beta,P^i\}_{i\geq 0}$, 
\item the relations $P^0=1$ and the so called {\it Adem relations},
\item the diagonal
  $\psi^*\co \mathcal{A}_{\top}^*\to\mathcal{A}_{\top}^*\otimes_{\p}
\mathcal{A}_{\top}^*$ described by the so called {\it Cartan formulae}.
\end{enumerate}
The standard reference for the ``algebraic'' constructions is
Voevodsky's paper \cite{voe-oper}. To construct the operations $P^i$,
we need a suitable and cohomologically rich enough simplicial sheaf
$\mathcal{B}$ and an homomorphism 
\begin{equation}\label{totale}
P\co H^{2*,*}(-,\p)\to H^{2p*,p*}(-\wedge\mathcal{B},\p)   
\end{equation}
called the {\it total power operation}. The ``suitability'' that the
simplicial sheaf shall satisfy is that
$H^{*,*}(\mathcal{X}\wedge\mathcal{B},\p)$ is free as left
$H^{*,*}(\mathcal{X},\p)$ module  over a basis $\{b_i\}_i$ for any
pointed simplicial sheaf $\mathcal{X}$. We
can then obtain individual motivic cohomology operations from $P$ by
defining classes $A_{i}(x)$  satisfying the equality
$P(x)=\sum_i A_{i}(x)b_i$ for any $x\in H^{*,*}(\mathcal{X},\p)$. In
analogy with the topological case, the pointed simplicial sheaf
$\mathcal{B}$ is going to be chosen among the ones of the kind $BG$
for a finite group (or, more generally, a group scheme) $G$. In 
particular, we will consider 
$G$ to be $S_p$, the group of permutations on $p$ elements.   
There are various models for the homotopy class of $BG$ in
$\mathcal{H}_\bullet(k)$. We are interested to one of them whose motivic 
cohomology can be computed more easily for the groups $G$ considered:
let $r\co G\to Gl_d(k)$ be a faithful 
representation of $G$ and $U_n$ the open subset in
$\mathbb{A}^{dn}_k$ where $G$ acts freely with respect of $r$. The
open subschemes $\{U_n\}_n$ fit into a direct system induced by the
embeddings $\mathbb{A}^{dn}\hookrightarrow \mathbb{A}^{d(n+1)}$
given by $(x_1,x_2,\ldots,x_n)\to (x_1,x_2,\ldots,x_n,
0)$. For instance, let $\mu_p$ defined as the group scheme
given by the kernel of the homomorphism
$(-)^p\co\mathbb{G}_m\to\mathbb{G}_m$. We
can take $d=1$ and $r\co \p\hookrightarrow Gl_1(k)=k^*$. It follows that
$U_n=\mathbb{A}_k^n-0$ in this case.    
\begin{definition}
$BG$ is defined to be the pointed sheaf $\colim_n(U_n/G)$, where
  $U_n/G$ is the quotient in the category of schemes and the colimit
  is taken in the category of sheaves.
\end{definition}
Of all the results of this section, the computation of the motivic
cohomology of $BG$ is the only one that will be  given a complete
proof. This is because of its importance and relevance in the
structure of the motivic cohomology operations we are going to study
later in this section. 

\begin{theorem}[see Voevodsky {{\cite[Theorem~6.10]{voe-oper}}}]
\label{class}
Let $k$ be any field and $\mathcal{X}$ a pointed simplicial sheaf over
$k$. Then, if $p$ is odd,
\begin{equation}
\begin{split}
H^{*,*}(\mathcal{X}\wedge
(B\mu_p)_+,\p)=\dfrac{H^{*,*}(\mathcal{X},\p)\llbracket u,v\rrbracket}{(u^2)}.
\end{split}
\end{equation}
If $p=2$,
\begin{equation}
H^{*,*}(\mathcal{X}\wedge
(B\mu_p)_+,\p)=\dfrac{H^{*,*}(\mathcal{X},\p)\llbracket u,v\rrbracket
}{(u^2-\tau v+\rho u)},
\end{equation}
where
\begin{enumerate}
\item $u\in H^{1,1}(B\mu_p,\p)$ and $v\in H^{2,1}(B\mu_p,\p)$;
\item $\rho$ is the class of $-1$ in $H^{1,1}(k,\p)\cong k^*/(k^*)^p$;
\item $\tau$ is zero if $p\neq 2$ or $\chara(k)=2$;
\item $\tau$ is the generator of
  $H^{0,1}(k,\mathbb{Z}/2)=\Gamma(\Spec k,\mu_2)=\mu_2(k)$ 
if $p=2$ and $\chara(k)\neq 2$.
\end{enumerate}
\end{theorem}

\begin{proof} It suffices to consider the case of $\mathcal{X}=X_+$
smooth scheme, essentially because of \cite[Lemma~1.16]{morvoe}
(see also \cite[Appendix~B]{io-for-grado}). To prove the statement of
the theorem as left $H^{*,*}(\mathcal{X},\p)$ modules, we are going
to use the existence of a cofibration sequence 
\begin{equation}\label{cof1}
(B\mu_p)_+\to (\mathcal{O}_{\mathbb{P}^\infty_k}(-p))_+\to
  \Th(\mathcal{O}(-p))\to S_s^1\wedge (B\mu_p)_+\cdots 
\end{equation} 
where $\mathcal{O}_{\mathbb{P}^\infty_k}(-1)$ is the dual of the
canonical line bundle over $\mathbb{P}^\infty_k$. 
By a vector bundle $V$ we sometimes mean its affinization
$\Spec(\mathrm{Symm}^*(V\,\wcheck{-}))$. This sequence is a consequence 
of the fact that
$B\mu_p=\mathcal{O}_{\mathbb{P}^\infty_k}(-p)-z(\mathbb{P}^\infty_k)$,
where $z$ is the zero section and $B\mu_p$ is constructed by means of
the representation $$\mu_p\hookrightarrow GL(\mathcal{O})=\mathbb{G}_m$$
(cf Voevodsky \cite[Lemma~6.3]{voe-oper}). Smashing the sequence \eqref{cof1}
with a smooth scheme $X_+$ we get a cofibration sequence
\begin{multline}\label{cof2}
X_+\wedge (B\mu_p)_+\longrightarrow X_+\wedge
(\mathcal{O}_{\mathbb{P}^\infty_k}(-p))_+\stackrel{f}{\longrightarrow}
 X_+\wedge \Th(\mathcal{O}(-p))\longrightarrow \\
S^1_s\wedge X_+\wedge (B\mu_p)_+\longrightarrow\cdots 
\end{multline}
which yields a long exact sequence 
\begin{multline}\label{bah}
\cdots\to \tilde{H}^{*-2,*-1}(X,A)\llbracket c\rrbracket
\stackrel{\hat{f}^*}{\longrightarrow} 
H^{*,*}(X,A)\llbracket c\rrbracket \stackrel{\alpha}{\longrightarrow} \\
H^{*,*}(X_+\wedge (B\mu_p)_+,A)\longrightarrow
\tilde{H}^{*-1,*-1}(X,A)\llbracket c\rrbracket \to\cdots
\end{multline}
for any abelian group $A$ and where $c\in H^{2,1}(\mathbb{P}^\infty,A)$ is
the first Chern class of $\mathcal{O}_{\mathbb{P}^\infty}(1)$. We
define $v\in H^{2,1}(B\mu_p,A)$ to be the image of $c$ and 
$u\in H^{1,1}(B\mu_p,\p)$ the class mapping to $1\in
H^{0,0}(\Spec k,\p)\llbracket c\rrbracket $ and restricting to $0$ in
$H^{1,1}(\Spec k,\p)\cong \{H^{*,*}(\Spec k,\p)\llbracket c\rrbracket \}^{1,1}$ via any
rational point $\Spec k\to B\mu_p$. The long exact sequence \eqref{bah}
splits in short exact sequences 
\begin{equation}\label{bah1}
0{\to}H^{*,*}(X,\p)\llbracket c\rrbracket {\to}H^{*,*}(X_+\wedge B\mu_p,\p){\to}
H^{*-1,*-1}(X,\p)\llbracket c\rrbracket {\to}0
\end{equation}
if $A=\p$ since $\hat{f}^*=0$, in that case. To see this, notice 
that $\hat{f}^*$ is the composition $z^*\circ f^*\circ t$ where 
$$t\co H^{*,*}(\mathbb{P}^\infty)\to
\tilde{H}^{*+2,*+1}(X_+\wedge \Th(\mathcal{O}(-p)))\tilde{=}H^{*+2,*+1}
(\Th(X\times\mathcal{O}(-p)))$$
is the Thom isomorphism (see Borghesi \cite[Corollary~1]{io-mor}) and
$$z\co X_+\wedge\mathbb{P}^\infty_+\to
  X_+\wedge\mathcal{O}_{\mathbb{P}^\infty_k} (-p)_+$$
is the zero section. This composition sends a
class $x$ to $x\cdot pc$ which is zero with $\p$ coefficients.
By means of the sequences \eqref{bah1} we prove the theorem as 
left $H^{*,*}(X,\p)$ modules.  
The exotic part of the theorem lies in the multiplicative structure of the
motivic cohomology at the
prime $2$, more specifically the relation $u^2=\tau v+\rho
u$. Since the multiplicative structure in motivic
cohomology is graded commutative (see Voevodsky
\cite[Theorem~2.2]{voe-oper}), at odd
primes we have $u^2=0$. Let $p=2$. The class $u^2$ belongs to the group
$H^{2,2}(B\mu_p,\p)$ which is isomorphic to
\begin{equation}
H^{0,1}(\Spec k,\p)v\oplus H^{1,1}(\Spec k,\p)u\oplus
H^{2,2}(\Spec k,\p)
\end{equation}
because of what we have just proved. By definition
of $u$, it restricts to $0$ on $\Spec k$, thus $u^2=xv+yu$ for
coefficients $x\in H^{0,1}(\Spec k,\due)$ and $y\in
H^{1,1}(\Spec k,\due)$. We wish to prove that $x=\tau$ and $y=\rho$ as
described in the statement of the theorem. 

\textbf{Proof that $y=\rho$}\qua
We
reduce the question to the group $H^{*,*}(\mathbb{A}^1-0,\due)$ by means
of the map 
\begin{equation}
\mathbb{A}^1-0\cong(\mathbb{A}^1-0)/\mu_2\to
\colim_i(\mathbb{A}^i-0)/\mu_2=B\mu_2
\end{equation} 
The class $u\in H^{1,1}(B\mu_p,\due)$ pulls
back to the generator $u_1$ of $$H^{1,1}(\aff-0,\due)\cong\due$$
because so does the similarly defined class $u_i\in
H^{1,1}(\mathbb{A}^i_k-0/\mu_2,\due)$ for each $i$: the sequence \eqref{cof1}
in this case reduces to $$\aff-0_+\to (\aff)_+\to \Th(\aff)=
\aff/(\aff-0)\to\dots$$ 
and the generator of  $H^{1,1}(\aff-0,\due)$ is precisely $u_1$, ie the
class coming from the Thom sheaf. The class $u$ pulls back to each
$u_i$ because of a lim$^1$ argument on the motivic cohomology of $B\mu_2$.
We have thus reduced the question to the following rather
curious lemma
\begin{lemma}
Let $u_1\in \tilde{H}^{1,1}(\aff-0,\due)\cong\due$ be the
generator. Then
$u_1^2=\rho u_1$ in $H^{2,2}(\aff-0,\due)$.
\end{lemma} 
One further reduction involves the use of the
canonical map $A\to\aff-0$ where $A$ is a sheaf whose motivic
cohomology is $\smash{\colim_{U\subset\aff\text{ open}}}H^{*,*}(U,\due)$. Because of
\cite{voe-iso}, $H^{i,j}(A,\due)$ is thus isomorphic to the higher Chow
group $CH^j(\Spec k(t),2j-i)$ and in 
particular $H^{i,i}(A,\due)\cong K_i^M(k(t))/(2)$. The sheaf $A$ can be
manufactured in many different non isomorphic ways as sheaves
although there is a canonical one. In any case, their class in 
$\mathcal{H}(k)$ coincides. Such class is known as the {\it homotopy
  limit} of the cofiltered category $\{U\}_{U\subset\aff\,open}$ and
denoted by holim$\,U$. It
comes equipped with a map $$\mathrm{holim}\,U\to U$$ for any 
such $U$ and the localizing sequence for motivic cohomology implies the
injectivity of the map $H^{*,*}(\aff-0,\due)\to
H^{*,*}(\mathrm{holim}\,U,\due)$. Therefore, it suffices to prove the
statement of the lemma in 
\begin{equation}
H^{2,2}(\mathrm{holim}\,U,\due)\cong CH^2(k(t),0)\otimes\due\cong 
K_2^M(k(t))/(2)
\end{equation}
With these
identifications we may represent $u_1$ by $t\in K_1^M(k(t))$  and show
that $t^2=-1\cdot t$, which is a known relation in
$K_2^M$.

\textbf{Proof that $x=\tau$}\qua
We have that $x\in H^{0,1}(\Spec k,\due)
=\mu_2(k)$ and this
group is zero if the characteristic of $k$ is $2$. If the characteristic is
different from $2$ then this group is $\due$ and thus it suffices to
show that $x\neq 0$. By contravariance of motivic cohomology, it is
enough to prove this statement up to replacing the base field $k$ with
a finite field extension. In particular, by assuming that
$\sqrt{-1}\in k$, we reduce the question to showing that $u^2\neq 0$,
since $\rho=0$ in this case. Passing to \'etale cohomology with $\due$
coefficients via the natural transformation from motivic cohomology
with $\due$ coefficients to \'etale cohomology with the same
coefficients, we are reduced to show that $u^2\neq 0$ in
$H^2_{\et}(B\mu_2,\due)$. All the long exact sequences involving
motivic cohomology we have used in this section are valid for \'etale
  cohomology as well. Since $u\in H^1_{\et}$, $\beta u=u^2$. Consider the
  long exact sequence 
\begin{multline}
\cdots \longrightarrow H^1_{\et}(B\mu_2,\mathbb{Z})
  \stackrel{\text{mod}\,2}{\longrightarrow}
  H^1_{\et}(B\mu_2,\due)\\
  \stackrel{\delta}{\longrightarrow}
  H^2_{\et}(B\mu_2,\due)\stackrel{2\cdot}{\longrightarrow}
  H^2_{\et}(B\mu_2,\due)\longrightarrow\cdots
\end{multline}  
we see that $\delta u\neq 0$. To see this we observe the following
diagram 
\begin{equation}
\xymatrix{ && H^1_{\et}(\Spec k,\due)\ar[d]^\alpha\\
H^1_{\et}(\Spec k,\mathbb{Z}) \ar[r]_\cong^\alpha \ar[rru]^{\text{mod}\,2}&
H^1_{\et}(B\mu_2,\mathbb{Z})\ar[r]^{\text{mod}\,2}&
H^1_{\et}(B\mu_2,\due)\ar[d]\\
&& H^0_{\et}(\Spec k,\due)}
\end{equation} 
where the vertical sequence of maps is \eqref{bah} with motivic
cohomology replaced by \'etale cohomology and is a short exact
sequence. The class $u\in
H^1_{\et}(B\mu_2,\due)$ is defined as the only class which maps to the
generator of $H^0_{\et}(\Spec k,\due)$ and is zero on the image of
$H^1_{\et}(\Spec k,\due)$, thus it is not in the image of the $\text{mod}\;2$
morphism. $\delta u\neq 0$ implies that
$\beta u\neq 0$ since $\delta u=v\in H^2_{\et}(B\mu_2,\mathbb{Z})$
which projects to the class $v$ with $\due$ coefficients. This
finishes the proof of \fullref{class}.
\end{proof}

\subsection{The dual algebra $\dual$}

Rather than considering motivic cohomology operations, we are
going to concentrate on their duals. Denote 
by $\am_m$ the $\p$  algebra of all the motivic cohomology operations and let 
$\am$ be a locally finite and free $H^{*,*}(\Spec k,\p)$ 
(simply written as $H^{*,*}$ from now on) submodule of it. Its
dual $\dual$ is the set of left $\hb$ graded module maps from $\am$ to $\hb$.
We are interested in the action of $\am$ on the motivic cohomology of
$B\mu_p$. Let $\theta\in\am$; its action on $H^{*,*}(B\mu_p,\p)$
is completely determined by $\theta(u^{e}v^i)$ for all $i$ and
$e\in\{0,1\}$, because of \fullref{class}. The module $\am$ we
are going to consider is in fact a $\p$ algebra and includes certain 
monomial operations denoted by $M_k$ for all nonnegative integers $k$ 
satisfying:
\begin{proposition}[see Voevodsky {{\cite[Lemma~12.3]{voe-oper}}}]
\label{proprieta}
\begin{enumerate}
\item $M_k(v)=M_k\beta(u)=v^{p^k}$, for all $k\geq 0$; 
\item if $\am\ni\theta\not\in \{M_k,M_k\beta,\,k\geq 0\}$ is a
  monomial, then $$\theta\cdot H^{*,*}(B\mu_p,\p)=0$$
\end{enumerate}
\end{proposition}  
This enables us to define canonical classes
$\xi_i\in\mathcal{A}_{2(p^i-1),p^i-1}$ for $i\geq 0$ and
$\tau_j\in\mathcal{A}_{2p^j-1,p^j-1}$ for $j\geq 0$ as the duals of
$M_k$ and $M_k\beta$, respectively. Monomials in the classes $\xi_i$ 
and $\tau_j$
are generators of $\dual$ as a left $H^{*,*}$ module. More
precisely, one first proves that
$\omega(I):=\tau_0^{\epsilon_0}\xi_1^{r_1}\tau_1^{\epsilon_1}\xi_2^r\cdots$
form a free $H^{*,*}$ module basis of
$\dual$, where $I=(\epsilon_0,r_1,\epsilon_1,r_2,\ldots)$ ranges over all the
infinite sequences of nonnegative integers $r_1$ and
$\epsilon_i\in\{0,1\}$  (cf \cite[Theorem~12.4]{voe-oper}). Then one 
derives the complete description of $\dual$:
\begin{theorem}[cf Voevodsky {{\cite[Theorem~12.6]{voe-oper}}}]
\label{duale} 
The graded left $H^{*,*}$ algebra  $\dual$ is (graded) commutative 
with respect to the first grading and is presented by generators
$\xi_i\in\mathcal{A}_{2(p^i-1),p^i-1}$ and
$\tau_i\in\mathcal{A}_{2p^i-1,p^i-1}$ and with relations
\begin{enumerate}
\item $\xi_0=1$; 
\item $\tau_i^2=
\begin{cases}
0 & \text{for}~ p\neq 2\\
\tau\xi_{i+1}+\rho\tau_{i+1}+\rho\tau_0\xi_{i+1} & \text{for}~ p=2
\end{cases}$
\end{enumerate}  
\end{theorem}   

\begin{remark}
The degree of the product $a\gamma$ between an element $a\in\hb$
  and $\gamma\in\dual$ is $(\gamma_1-a_1,\gamma_2-a_2)$.
\end{remark}

\subsection{The algebra $\am$}

To get this kind of information on $\dual$ we
really have to be more specific on the algebra $\am$ we are
considering. As mentioned earlier, Voevodsky \cite[Section~5]{voe-oper}
first defined a total power operation 
$$P\co H^{2*,*}(-,\p)\to H^{2p*,p*}(-\wedge\mathcal{B},\p)$$
Taking $\mathcal{B}$ to be $BS_p$, this becomes a morphism
\cite[Theorem~6.16]{voe-oper}
\begin{equation}
\begin{split}
P\co H^{2*,*}(X,\p)\to 
\begin{cases}
\{H^{*,*}(X,\p)\llbracket c,\!d\rrbracket
  /(c^2{=}\tau d{+}\rho c)\}^{4*,2*}, & p=2\\
\{H^{*,*}(X,\p)\llbracket c,\!d\rrbracket
  /(c^2)\}^{2p*,p*}, & p\neq 2
\end{cases}
\end{split}
\end{equation}
for classes $c\in H^{2p-3,p-1}(BS_p,\p)$ and $d\in
H^{2p-2,p-1}(BS_p,\p)$ and any simplicial sheaf $X$. 
The unusual relation in the motivic cohomology of $BS_p$ at the 
prime $p=2$ is consequence of the one in $H^{*,*}(B\mu_2,\due)$ 
described in \fullref{class}.
Thus, in analogy to the topological case, we can define operations
$P^i$ and $B^i$ by the equality
\begin{equation}\label{operazioni}
P(w)=\sum_{i\geq 0}P^{i}(w)d^{n-i}+B^{i}(w) cd^{n-i-1}
\end{equation}
for $w\in H^{2n,n}(X,\p)$. Then one proves \cite[Lemma~9.6]{voe-oper}
that $B^i=P^i\beta$. For an arbitrary class $x\in H^{*,*}(X,\p)$, we
define
$$P^i(x):=\sigma_s^{-m_1}\sigma_t^{-m_2}
(P^i(\sigma_s^{m_1}\sigma_t^{m_2}(x)))$$    
where $\sigma_s^{m_1}\sigma_t^{m_2}(x)\in H^{2*,*}(X,\p)$ for some $*$
and $\sigma$ are the suspension isomorphisms we introduced at the
beginning of this section. As left $\hb$ module we let 
$\am$ to be $\hb\otimes_{\p}\p\,\langle \beta,P^i\rangle$, where 
$\p\,\langle \beta,P^i\rangle$ is the $\p$ subalgebra (of all the
motivic cohomology operations $\am_m$) generated by $\beta$ and $P^i$. It
turns out that $\am$ is a free left $\hb$ module (cf
\cite[Section~11]{voe-oper}). Notice that there is canonical embedding
$i\co \hb\hookrightarrow\am$ sending $a\in\hb$ to the operation
$ax=\pi^*a\cup x$, where $\pi\co X\to \Spec k$ is the structure morphism.
The image $i(\hb)$ does not belong to the center of $\am$: the
multiplication in $\am$ is the composition of the cohomology
operations, therefore $(\theta a)x=\theta(ax)=\theta(a\cup
x)=\sum_k\smash{\theta_k'\act a}\otimes\theta_k''(x)$ where
$\sum_k\theta_k'\otimes\theta_k''$ is the image of $\theta$ through
the morphism $\hat{\psi}^*\co \am_m\to\am_m\otimes_{\hb}\am_m$ induced by the
multiplication (cf \cite[Section~2]{voe-oper})
$$\hat{\psi}\co  K(A(j_1),i_1)\wedge K(A(j_2),i_2)\to K(A(j_1+j_2),i_1+i_2)$$
of the motivic Eilenberg--MacLane simplicial
sheaf defined at the end of \fullref{puggia}. Here
$\theta\act a$ means the motivic cohomology operation $\theta$ acting
on $a$ seen as a cohomology class of a scheme. Although, the 
$\p$ vector space $\hb$ does not lie in the center of $\am$, the 
commutators are sums of terms of monomials of the kind $aP^I$. Notice
that $\am$ is very much not (even graded) commutative: the relations
between the products of $P^i$ are called {\it Adem relations} and are
very complicated already in the topological case. For the expression
in the algebraic situation see 
\cite[Theorems~10.2,10.3]{voe-oper}. It
turns out that the classes $M_k$ are $P^{p^{k-1}}P^{p^{k-2}}\cdots
P^pP^1$. This fact is needed to prove that $\dual$ is a free left $\hb$
module on the classes $\omega(I)$. 

To understand the product of elements in the dual, we
need to endow $\am$ of an extra structure: a left $\hb$ module map
$\psi^*\co \am\to \am\otimes_{\hb}\am$ called {\it diagonal} or {\it
  comultiplication}. We wish to define $\psi^*$ to be the restriction
of $\hat{\psi}^*$ to $\am$. In order to do this we have to show that
the image of $\hat{\psi}^*$ is contained in $\am$ when the domain is
restricted to $\am\otimes_{\hb}\am$. This can be checked by
explicitely computing the value of the 
the total power operation $P$ on the exterior product of motivic
cohomology classes $x$ and $y$ of a simplicial sheaf $X$ 
and the property that $P(x\wedge
y)=\Delta^*(P(x)\wedge P(y))$ (Voevodsky~\cite[Lemma~5.9]{voe-oper}), where
$\Delta\co X\to X\times_kX$ is the diagonal. 
We include here the complete expression of
$\psi^*$ since the one in \cite[Proposition~9.7]{voe-oper} has
a sign error at the prime $2$:
\begin{theorem}\label{cartan}
Let $u$ and $v$ be motivic cohomology classes. Then for $p$ odd we
have
\begin{equation}
\begin{aligned}
P^i(u\wedge v)&=\sum_{r=0}^iP^r(u)\wedge P^{i-r}(v)\\
\beta(u\wedge v)&=\beta(u)\wedge v+(-1)^{\mathrm{first\,deg}(u)}u\wedge\beta v.
\end{aligned}
\end{equation}
If $p=2$ then
\begin{align}
\Sq^{2i}(u{\wedge} v)&=\sum_{r=0}^i\Sq^{2r}(u){\wedge}
\Sq^{2i-2r}(v)+\tau\sum_{s=0}^{i-1} \Sq^{2s+1}(u){\wedge}
\Sq^{2i-2s-1}(v)\\
\Sq^{2i+1}(u{\wedge} v)&=\sum_{r=0}^i\bigl(\Sq^{2r+1}(u){\wedge}
\Sq^{2i-2r}(v)+\Sq^{2r}(u){\wedge}
\Sq^{2i-2r+1}(v)\bigr)\\[-2ex]
&\hskip 125pt+\rho\sum_{s=0}^{i-1}\Sq^{2s+1}(u) {\wedge}
\Sq^{2i-2s-1}(v).\notag
\end{align}
\end{theorem}

As usual, the strange behaviour of the motivic cohomology of $B\mu_2$ at
the prime $2$ results in more complicated description of $\psi^*$ at
the same prime. However, even in this case, $\psi^*$ is associative
and commutative. In such a situation we can define a product structure
on $\dual$ by means of the transposed map of $\psi^*$: let $\psi_*$ be
the map $\dual\otimes_{\hb}\dual\to\dual$ such that
\begin{equation}
\langle \psi_*(\eta_1\otimes\eta_2), \theta\rangle=\langle
\eta_1\otimes\eta_2,
\psi^*(\theta)\rangle=\sum\eta_1(\theta')\eta_2(\theta'') 
\end{equation}
for any $\theta\in\am$. 
This makes $\dual$ in an
associative and graded commutative (with respect to the first grading) 
$\hb$ algebra. 

\textbf{Proof of \fullref{duale} (sketch)}\qua
Since $\tau_i$'s first degrees are odd, we have
$\tau_i^2=0$ at odd primes because of graded commutativity. At the
prime $2$ the situation is more complicated. The key trick goes back to
Milnor's original paper \cite{milnor}: we define a morphism 
$\lambda\co  H^{*,*}(X,\p)\to H^{*,*}(X,\p)\otimes_{\hb}\dual$ as
\begin{equation}
\lambda(x)=\sum_{i}e_i(x)\otimes e^i
\end{equation} 
where $\{e_i\}$ are a basis of $\am$ over $\hb$ and $e^i$ the dual
basis. The morphism $\lambda$ does not depend on the choice of $e_i$,
thus we can take $\{e_i\}=\{\beta, P^I\}$ for all indices
$I=(i_1,i_2,\ldots)$. Because of \fullref{proprieta}, we can
explicitely compute the morphism $\lambda$ in the case of $X=B\mu_p$,
at least on the classes $u$ and $v$: $\lambda(u)=u\otimes 1+\sum_i
v^{p^i}\otimes \tau_i$ and $\lambda(v)=v\otimes 1+\sum_i
v^{p^i}\otimes \xi_i$. We now need some equality in which $\tau_i^2$
will appear. We can take such expression to be $\lambda(u)^2$. We
are interested to compare it with $\lambda(u^2)=\lambda(\tau v+\rho u)$.
The crucial property of $\lambda$ we use here is multiplicativity
with respect to the cup product; than we derive the relation
$\tau_i^2$ by setting the homogeneous components of the
expressions to be equal. Just like for the case of $\am$, the dual 
algebra $\dual$ has two actions of $\hb$: the left one $a\otimes\xi\to
\smash{a\stackrel{l}{\cdot}\xi}$ where $a\xi$ is
defined by $$\langle \smash{a\stackrel{l}{\cdot}\xi},\theta\rangle=\langle\xi,
a\theta\rangle=a\langle\xi,\theta\rangle$$
The right action is $\xi\otimes a\to \xi\stackrel{r}{\cdot} a$, with 
\begin{equation}
\langle\xi\stackrel{r}{\cdot} a, \theta\rangle=\langle\xi,\theta a\rangle
=\langle\xi,\sum_k(\theta_k'\act a)\otimes\theta_k''\rangle
=\sum_k\theta_k'\act a\langle\xi,\theta_k''\rangle
\end{equation}
We form the tensor product $\dual\otimes_{\hb}\dual$ according with
  these actions.
To complete the picture, we shall compute the diagonal $\phi_*\co \dual\to 
\dual\otimes_{\hb}\dual$ (the left factor is understood to be endowed 
with the right
action of $\hb$ and the right factor with the left action), defined 
as the transposed of the
multiplication in $\am$: if $\gamma\in\dual$ then its action on a
product of operations $\alpha'$, $\alpha''$ is expressed by
$$\langle\gamma,
\alpha'\alpha''\rangle=\sum_i\langle\gamma'_i,\alpha'\rangle\langle
\gamma_i'',\alpha''\rangle$$
for some classes $\gamma'$ and $\gamma''$. We then define
$\phi^*(\gamma)=\sum_i \gamma'_i\otimes\gamma_i''$. To compute what
$\phi^*(\tau_i)$ and $\phi^*(\xi_i)$ are, we use a strategy already
employed: find some equality involving the motivic cohomology classes
$u$ and $v$ of $B\mu_p$ containing expressions like $\langle \tau_i,
\alpha'\alpha''\rangle$ and $\langle \gamma_i',\alpha'\rangle\langle 
\gamma''_i,\alpha''\rangle$ as coefficients of certain monomials and
then get the result by setting equal the homogeneous components of the
expressions. We can write the action of a motivic cohomology operation
$\theta$ on a class $x\in H^{*,*}(X,\p)$ in such a way to have
elements of $\am$ appearing: 
\begin{equation}
\theta(x)=\sum_i\langle\xi_i,\theta\rangle M_i(x)+\langle
\tau_i,\theta\rangle M_i\beta(x)
\end{equation}
Once again we take $X$ to be $B\mu_p$ since we know everything about
it and prove the equalities
\begin{equation}\label{rottura}
\begin{split}
\theta(u^{p^n})&=\langle\xi_0,\theta\rangle u^{p^n}+\sum_i\langle
\tau_i^{p^n},\theta\rangle v^{p^{i+n}}\\
\theta(v^{p^n})&=\sum_i\langle\xi_i^{p^n},\theta\rangle v^{p^{i+n}}
\end{split}
\end{equation}
by means of \fullref{proprieta}. The expressions we are looking
for are $\gamma\theta(u)$ and $\gamma\theta(v)$  for $\gamma$,
$\theta\in\am$. Each of them can be written in two ways: as
$(\gamma\theta)(u)$ and as $\gamma(\theta(u))$. Using equations
\eqref{rottura} we obtain two equalities between polynomials in $u$
and $v$. Equality between their coefficients give the following
formulae:
\begin{proposition}
\begin{equation}
\begin{split}
\phi_*(\tau_k)&=\tau_k\otimes 1+\sum_i\xi_{k-i}^{p^i}\otimes\tau_i\\
\phi_*(\xi_k)&=\xi_{k-i}^{p^i}\otimes\xi_i
\end{split}
\end{equation}
\end{proposition}
The computation of the diagonal of the dual motivic Steenrod algebra
is crucial to find relations between certain special classes $Q_i$ 
in $\am$ we are now going to define. 
Let $E=(\epsilon_0,\epsilon_1,\ldots)$ and $R=(r_1,r_2,\ldots)$. and
$\tau(E)\xi(R)=\prod_{i\geq 0}\tau_i^{\epsilon_i}\prod_{j\geq
  1}\xi_j^{r_j}$. 
\begin{definition}
The following notation will be used:
\begin{enumerate}
\item $(r_1,r_2,\ldots,r_n)\in\am$ will denote 
the dual class to $\xi_1^{r_1}\xi_2^{r_2}\cdots\xi_n^{r_n}$;
\item if $E=(\epsilon_0,\epsilon_1\cdots,\epsilon_m)$, where
  $\epsilon_i\in\{0,1\}$, then $Q_E$ will denote the dual to  
$\tau_0^{\epsilon_0}\tau_1^{\epsilon_1}\cdots\tau_m^{\epsilon_m}$;
\item $Q_i$ will be the dual to $\tau_i$.
\end{enumerate}
\end{definition}
 
The most important properties of $Q_t^{\top}$, the topological
cohomology operations defined exactly as $Q_t$, are: 
\begin{enumerate}
\item $(Q_t^{\top})^2=0$,
\item $\psi^*(Q_t^{\top})=Q_t^{\top}\otimes 1+1\otimes Q^{\top}_t$, ie
$Q_t^{\top}$ is primitive and
\item If $M$ is a complex manifold with tangent bundle $T_M$ 
is such that all the characteristic numbers are divisible by $p$ and 
the integer $\deg(s_{p^n-1}(T_M))$ (see the definition of it 
given just after \fullref{vaff}) is not divisible by $p^2$, then 
$Q_t^{\top} t_\nu\neq 0$ in the cone of the map $S^{2t}\to \Th(\nu)$,
  coming from the Thom--Pontryagin construction, where $\nu$ is the normal
  bundle, with complex structure, of some embedding 
  $M\hookrightarrow\mathbb{R}^N$ for $N$ large enough and $t_\nu$ is
  its Thom class in $H^{2m,m}(\Th(\nu),\p)$. 
\end{enumerate}
We are interested in the operations $\{Q_t\}_t$ because they appear 
in the proof of
\fullref{nullo}, that is crucial for the Voevodsky's program
to the Bloch--Kato conjecture, in which property (3) is used. It turns 
out that the operation $Q_t$
satisfies property (1) (use the coproduct of $\dual$), but fails to
satisfy property (2) at the prime $2$, if $\sqrt{-1}\not\in
k$ for $t>0$. Nonetheless, property (3) holds at any prime, and this is what
we need for the application in \fullref{nullo}. To compute the
coproduct $\psi^*Q_t$ we use its adjointness with the multiplication
of the dual algebra $\dual$. This results in:
\begin{proposition}\label{qt}
\begin{enumerate}
 \item if $p$ is odd, $Q_t$ are primitive;
\item if $p=2$
\begin{equation}
\psi^*(Q_t)=Q_t\otimes 1+1\otimes Q_t+\sum_{h=1}^t\hskip
7pt\rho^h\hskip 9pt
\Bigl(\hskip -1cm\sum_{\substack{I,J\\I\cup J
=\{t-h,t-h+1,\ldots,t-1\}\\I\cap J=\{t-h\}}}\hskip -1cm Q_I\otimes Q_J\Bigr)
\end{equation}
\end{enumerate}
\end{proposition}
To prove property (3) we use a result of Voevodsky
\cite[Corollary~14.3]{voe-oper}:
\begin{proposition}\label{porc}
Let $X$ be a scheme and $V$ a rank $m$ vector bundle over $X$. If
$t_V$ is the Thom class of $V$ in $H^{2m,m}(\Th(V),\p)$ then
$$(0,\ldots,0,\stackrel{n}{1})(t_V)=
s_{p^n-1}(V)\in
H^{2m+2(p^n-1),m+p^n-1}(\Th(V),\p)$$ 
\end{proposition}
Given this result, to prove property (3) we use the equality
\begin{equation}
Q_t=[Q_0,(0,\ldots,0,\stackrel{t}{1})]=Q_0(0,\ldots,0,\stackrel{t}{1})-
(0,\ldots,0,\stackrel{t}{1})Q_0 
\end{equation}
which happens to hold at any prime $p$. The general formulae for the
commutators can be found in \cite[Corollary~4]{io-mor} and, at the
prime $2$, differ from their topological counterparts. The degree of
$Q_0$ is $(1,0)$, so $Q_0 t_V=0$ if $X$ is smooth by the Thom
isomorphism and degree considerations. Thus, we are reduced
to prove that $Q_t(t_\nu)=Q_0(0,\ldots,0,1)(t_\nu)\neq
0$, where in this algebraic case $\nu$ is a ``normal'' bundle suitably
defined for this purpose (see Voevodsky \cite{on2tors} or Borghesi
\cite{io-for-grado}).
By \fullref{porc} and because we know that $Q_0$ is the
Bockstein we are reduced to show that
$s_{p^n-1}(\nu)$ is nonzero, (we know
that because it is the opposite of the same characteristic number of
$T_X$ which has nonzero degree by assumption) and that it is not the reduction 
modulo $p$ of a class in $H^{2*,*}(\Th(\nu),\mathbb{Z}/{p^2})$. This
requires a short argument using the assumption on the degree of
$s_{p^n-1}(T_X)$.

A natural question to ask is whether $\am_m=\am$, that is if all the
bistable motivic cohomology operations are those of the $\p$ vector space
$\am$. In the case the characteristic of the base field is zero, this
result has been claimed few times, the latest of which is in
\cite[Lemma~2.2]{voe-disp}. In the general case, the canonical inclusion
$\am\hookrightarrow\am_m$ makes the latter a graded left $\am$
module. Moreover, when motivic cohomology is representable in 
$\mathcal{H}(k)$, eg if $k$ is a perfect field, we know that such 
inclusion is split in the category of graded left $\am$ modules 
(combine \cite[Remark~5.2]{io-mor} with \cite{io-perf}). In fact,
$\am_m$ is a graded free left $\am$ module because of
\cite[Theorem~4.4]{milmoore}.

\subsection{Final remarks}

\fullref{finale*} asserts the representability of motivic
cohomology groups in the category $\mathcal{H}(k)$. In this last
section we will mention to two more interesting aspects, which we have
not planned to cover thoroughly in this manuscript. We have been
pretty vague about motivic cohomology since we were just interested in
representing it in a suitable category simply as functor with values 
in abelian groups. However, the reader should know that such
cohomology theory is expected to have several properties encoded by the
{\it Beilinson Conjectures}. In particular, one of them states that
the motivic cohomology of a smooth scheme $X$ should be isomorphic to 
the {\it Zariski} hypercohomology of $X$ with coefficients in some
complex of sheaves. In \fullref{motcoho} motivic cohomology    
has been defined to the {\it Nisnevich} hypercohomology of $X$ with
coefficients in a complex of sheaves $A(j)$. In general, these groups
differ from the ones obtained by taking the Zariski
hypercohomology. An important result of Voevodsky in \cite{libro}
states that, if $X$ is a smooth scheme over a perfect field, and
$D_*$ is a complex of Nisnevich sheaves {\it with transfers} and 
with homotopy invariant homology sheaves then
$\mathbb{H}_{\Nis}^*(X,D)\cong H_{\Zar}^*(X,D)$. The homotopy invariance
of the homology sheaves of a complex of Nisnevich sheaves with
transfers is strictly related to the concept of the complex being
$\aff$ local. In fact, if the base field is perfect, for such
complexes the two notions are equivalent. More tricky is the condition
for a sheaf to have transfers. These supplementary data on sheaves 
relates motivic cohomology to
algebraic cycles and to the classical theory of motives (see Voevodsky's
use of Rost's results on the motif of a Pfister quadric in
\cite{on2tors}). These considerations suggest that the sheaves of the
complexes $A(j)$ should have transfers and the complexes have homotopy
invariant homology sheaves. Now,
congruence \eqref{sgiang} in particular implies
``representability'' of motivic cohomology in the category
$\mathcal{D}_+(\mathcal{N}_k)$ and \fullref{finale*} in the
category  $\mathcal{D}_+^{\smash{\aff}}(\mathcal{N}_k)$. To preserve this
property even when switching to the definition of motivic cohomology as Zariski
hypercohomology (as opposed to Nisnevich hypercohomology), Voevodsky's
theorem indicates we should work just with sheaves with transfers and
complexes with homotopy invariant homology sheaves. The
category of Nisnevich sheaves of abelian groups with transfers
$\mathcal{N}_k^{\tr}$ is an
abelian subcategory of 
$\mathcal{N}_k$ with enough injectives. Thus, we can consider its
derived category of bounded below chain complexes. It is nontrivial to
show that motivic cohomology is representable in 
$\mathcal{D}_+(\mathcal{N}_k^{\tr})$. In this case the functor
$\mathbb{Z}[X]$ is replaced by an appropriate ``free'' sheaf with transfers
$\mathbb{Z}_{\tr}[X]$. The further restriction to complexes with
homotopy invariant homology sheaves is more straightforward as it may
be encoded in a localizing functor, making the procedure very similar
to the $\aff$ localization of $\mathcal{H}_s(k)$. The outcome is a
category denoted by $DM_+(k)$. For the details of these
constructions, see Suslin--Voevodsky \cite[Theorem~1.5]{sus-voe}.
Using the terminology introduced in that paper, the sheaf $A(j)$ is
then defined to be
$$\smash{\mathbb{Z}(1)\overset{L}{\underset{\tr}{\otimes}}\stackrel{j}{\cdots}
\overset{L}{\underset{\tr}{\otimes}}\mathbb{Z}(1)\overset{L}
{\underset{\tr}{\otimes}}A}$$
where
$\mathbb{Z}(1):=C_*(\mathbb{Z}_{\tr}[\mathbb{G}_m])[-1]$ and $C_*(-)$
is the ``homotopy invariant homology sheaves'' localizing functor. 
An advantage of the category $DM_+(k)$ over $\mathcal{H}(k)$ is that
it is triangulated, whereas the latter it is not although the latter has
fibration and cofibration sequences. To overcome this disadvantage
there are some procedures to {\it stabilize} $\mathcal{H}(k)$ and make
it triangulated in a way that cofibration sequences become exact
triangules. Once again it is possible to prove that representability
of motivic cohomology is preserved, always under the assumption of
perfectness of the base field (see \cite{io-for-grado}). This 
makes the {\it stable homotopy
category of schemes} an effective working place for using homotopy
theory on algebraic varieties and consistently exploited in 
\cite{io-mor} as well as \cite{io-for-grado}.   


\nocite{mazza-weibel-voe}

\bibliographystyle{gtart}
\bibliography{link}

\end{document}